\documentclass[a4paper]{amsart}
\usepackage[dvips]{hyperref} 
\usepackage{amsmath,amssymb,amsthm,amscd, amsfonts}
\usepackage{bbold}
\usepackage{psfrag}
\usepackage{graphicx}
\usepackage{subfigure}
\usepackage{leftidx}
\usepackage[all]{xy}
\usepackage[latin1]{inputenc}
\usepackage{enumerate}

\usepackage[dvipsnames]{xcolor}

\newcommand{\blabla}[1]{\quad\text{#1}\quad}
\newcommand{\Comp}{\mathbb{C}}

\definecolor{ABColor}{cmyk}{0.1,0.8,0,0.5}
\definecolor{Qcolor}{cmyk}{0,1,1,0.1}

\newcommand\Lim{``\lim"}

\newtheorem{theorem}{Theorem}[section]
\newtheorem{corollary}[theorem]{Corollary}

\newtheorem{proposition}[theorem]{Proposition}
\newtheorem{lemma}[theorem]{Lemma}

\theoremstyle{definition}
\newtheorem{definition}[theorem]{Definition}

\newtheorem{example}[theorem]{Example}

\newcommand{\kk}{\Bbbk}
\newcommand{\kt}{$\Bbbk$\nobreakdash-\hspace{0pt}}

\newcommand{\ff}{F}

\theoremstyle{remark}
\newtheorem{remark}[theorem]{Remark}

\numberwithin{equation}{section}\theoremstyle{plain}

\newcommand{\iso}{\stackrel{\sim}{\longrightarrow}}
\newcommand{\toto}{\longrightarrow}
\newcommand{\into}{\hookrightarrow}

\newcommand{\C}{{\mathcal C}}
\newcommand{\E}{{\mathcal E}}
\newcommand{\A}{{\mathcal A}}
\newcommand{\B}{{\mathcal B}}
\newcommand{\D}{{\mathcal D}}
\newcommand{\T}{{\mathcal T}}

\newcommand{\U}{{\mathcal U}}
\newcommand{\Z}{{\mathcal Z}}

\newcommand{\Ll}{{\mathcal L}}
\renewcommand{\1}{\textbf{1}}

\newcommand{\KER}{\mathfrak{Ker}}

\newcommand{\TT}{\mathbb T}

\newcommand\car{\operatorname{char}}
\newcommand\id{\operatorname{id}}

\newcommand\Aut{\operatorname{Aut}}
\newcommand\co{\operatorname{co}}
\newcommand\cha{\operatorname{char}}

\newcommand\End{\operatorname{End}}

\newcommand\infl{\operatorname{infl}}
\newcommand\Res{\operatorname{Res}}
\newcommand\Ind{\operatorname{Ind}}

\newcommand\Hom{\operatorname{Hom}}
\newcommand\Pic{\operatorname{Pic}}
\newcommand\rep{\operatorname{rep}}

\newcommand\FPdim{\operatorname{FPdim}}
\newcommand\FPind{\operatorname{FPind}}
\newcommand\vect{\operatorname{vect}}

\newcommand\coker{\operatorname{coker}}

\newcommand\Pro{\operatorname{Pro}}

\newcommand\MOD[2]{\mathrm{mod}_{#1}{#2}}

\newcommand\ti{\mbox{-}\!}

\newcommand\Rep[1]{#1\ti\operatorname{mod}}
\newcommand\CoRep[1]{\operatorname{comod}\!\mbox{-} #1}
\newcommand\COREP[1]{\operatorname{Comod}\!\mbox{-} #1}

\newcommand{\ldual}[1]{\leftidx{^\vee}{\!#1}{}}
\newcommand{\rdual}[1]{{#1}^\vee}
\newcommand{\dual}[1]{{#1}^*}

\newcommand{\opp}{\mathrm{op}}

\newcommand{\prettydef}[1]{\left\{\begin{array}{ccl}#1\end{array}\right.}

\newcommand{\implique}{$\Rightarrow$}
\newcommand{\ssi}{$\Leftrightarrow$}

\begin{document}
\title[]{Exact sequences of tensor categories}
\author{Alain Brugui\`{e}res}
\author{Sonia Natale}
\address{Alain Brugui\`{e}res: D\' epartement de Math\' ematiques.
\; Universit\' e Montpellier II. Place Eug\` ene Bataillon. 34 095
Montpellier, France} \email{bruguier@math.univ-montp2.fr
\newline \indent \emph{URL:}\/ http://www.math.univ-montp2.fr/~bruguieres/}
\address{Sonia Natale: Facultad de Matem\'atica, Astronom\'\i a y F\'\i sica.
Universidad Nacional de C\'ordoba. CIEM -- CONICET. Ciudad
Universitaria. (5000) C\'ordoba, Argentina}
\email{natale@famaf.unc.edu.ar
\newline \indent \emph{URL:}\/ http://www.famaf.unc.edu.ar/$\sim$natale}

\thanks{The work of the second author was partially supported by CONICET,
ANPCyT, SeCYT--UNC and Alexander von Humboldt Foundation}


\date{\today.}

\begin{abstract}  We introduce the notions of normal tensor functor and exact sequence of tensor
categories. We show that exact sequences of tensor categories generalize strictly exact sequences of Hopf algebras as defined by Schneider, and in particular, exact sequences of (finite) groups. We classify exact sequences of tensor categories $\C' \to \C \to \C''$ (such that $\C'$ is finite) in terms of normal faithful Hopf monads on $\C''$ and also, in terms of self-trivializing commutative algebras in the center of $\C$.
More generally, we  show that, given any dominant
tensor functor $\C \to \D$ admitting an exact (right or left)
adjoint  there exists a canonical commutative algebra $(A,\sigma)$
in the center of $\C$ such that $\ff$ is tensor equivalent to the
free module functor $\C \to \MOD{\C}{(A,\sigma)}$, where
$\MOD{\C}{(A,\sigma)}$ denotes the category of $A$-modules in $\C$
endowed with a monoidal structure defined using $\sigma$.
We re-interpret equivariantization under a finite group action on a tensor
category and, in particular, the modularization construction, in terms of exact sequences, Hopf monads and commutative central algebras.
As an application, we prove
that a braided fusion category whose dimension is odd and
square-free is equivalent, as a fusion category, to the
representation category of a group.
\end{abstract}

\maketitle

\setcounter{tocdepth}{1} \tableofcontents

\section*{Introduction}

Tensor categories are abelian categories over a field $\kk$ having finite dimensional $\Hom$ spaces and objects of finite length, endowed with a rigid (or autonomous) structure,
that is a monoidal structure with duals, such that the monoidal tensor product is \kt bilinear and the unit object $\1$ is simple ($\End(\1)=\kk$).  A fusion category is a split semisimple tensor category having finitely many isomorphism  classes of simple objects. A tensor functor is a strong monoidal \kt linear functor between tensor categories.

In this paper, we introduce and study the notion of exact sequence of tensor categories, defined as follows. Let  $F : \C \to \D$ be a tensor functor. Then $F$ is \emph{dominant} if any object $Y$ of $\C'$ is a subobject of $F(X)$ for some  $X$ in $\C$. It is  \emph{normal} if any object $X$ of $\C$ admits a subobject $X_0$ such that $F(X_0)$ is the largest subobject of $F(X)$ which is trivial, that is isomorphic to $\1^n$. Denote by $\KER_F$ the full subcategory of $\C$ whose objects have a trivial image under $F$.

 An exact sequence of tensor categories is a diagram of tensor functors
 \begin{equation*}\label{exact}\xymatrix{\C' \ar[r]^i & \C \ar[r] ^\ff & \C''}\end{equation*}
such that
\begin{enumerate}[(a)]
\item $\ff$ is normal;
\item  $\ff$ is dominant;
\item $i$ induces an equivalence between $\C'$ and $\KER_\ff \subset \C$.
\end{enumerate}

This notion extends the notion of strictly exact sequence of Hopf algebras introduced by  Schneider in~\cite{schn2}. Indeed any strictly exact sequence of Hopf algebras:
$$K \overset{i} \toto H \overset{p} \toto  H'$$  gives rise to
an exact sequence of tensor categories:
\begin{equation*}\xymatrix{\CoRep{K} \ar[r] & \CoRep{H} \ar[r] &\CoRep{H'},}\end{equation*}
and, if $H$ is finite-dimensional, to a second exact sequence of tensor categories:
\begin{equation*}\xymatrix{\Rep{H'} \ar[r] & \Rep{H} \ar[r] &\Rep{K}.}\end{equation*}

For instance, an exact sequence of groups $\xymatrix@C=1.5em{1 \ar[r] &G' \ar[r]& G \ar[r]& G'' \ar[r] &
1}$ yields an exact sequence of tensor categories:
\begin{equation*}\label{ext-groupcats}\xymatrix{\C(G') \ar[r]& \C(G) \ar[r]& \C(G''),}\end{equation*}
(where $\C(G)$ denotes the tensor category of $G$-graded vector spaces) and if $G$ is finite,  a second exact sequence of tensor categories:
\begin{equation*}\label{ext-groupcats}\xymatrix{\rep G'' \ar[r]& \rep G \ar[r]& \rep G'.}\end{equation*}


In particular, we study exact sequences of fusion categories, and we show that the Frobenius-Perron dimension is multiplicative, that is, given an exact sequence of fusion categories
$\C' \to \C \to \C''$, we have  $$\FPdim \C = \FPdim \C' \,\FPdim \C''.$$
We show that an exact sequence of pointed categories is classified by an exact sequence of finite groups $1 \to G' \to G \to G'' \to 1$ together with a cohomology class
$\alpha \in H^3(G'',\kk^\times)$.  Also, generalizing a well-known result for semisimple Hopf algebras, we
show that a dominant tensor functor $\ff: \C
\to \C''$ of Frobenius-Perron index $2$ between fusion categories is  normal, and therefore gives rise to an exact sequence: $\rep \mathbb{Z}_2 \to \C \to \C''$.


Can we interpret an exact sequence of tensor categories $$(\E)\quad \C' \toto \C \overset{F}\toto \C''$$
in terms of `algebraic' data on $\C''$, or on $\C$? For technical reasons, we assume that $\C'$ is finite, that is, $F$ has adjoints. Then we show that $(\E)$ is encoded by a certain Hopf monad on $\C$, and also, if the right adjoint of $F$ is exact, by a certain commutative algebra in the center $\Z(\C)$ of $\C$.

A Hopf monad on a rigid  category $\D$ (as defined in
\cite[3.3]{bv})  is an algebra $T$ in the monoidal category
$\End(\D)$ of endofunctors of $\D$, which is also a comonoidal
functor in a compatible way, and possesses left and right
antipodes.

An exact sequence of tensor categories over a field $\kk$:
$$\C' \overset{f}\toto \C \overset{F} \toto \C''\leqno{(\E)}$$
 defines a fiber functor $\omega=\Hom(\1,Ff) : \C' \to \vect_\kk$,
hence by Tannaka reconstruction a Hopf algebra $H=L(\omega)$, called the induced Hopf algebra of $(\E)$, such that $\C' \simeq \CoRep{H}$.
The Hopf algebra $H$ is finite-dimensional if and only if the tensor functor $F$ has adjoints.

A \kt linear right exact Hopf monad $T$ on a tensor category $\C$ is
\emph{normal} if $T(\1)$ is a trivial object. If $T$ is such a Hopf monad, and if in addition $T$ is faithful, then it gives rise to an exact sequence of tensor categories
$$\CoRep{H} \to \C^T \to \C,$$
where $H$ is the induced Hopf algebra of $T$, isomorphic to $\Hom(T(\1),\1)$.

We show that, given tensor categories $\C'$ and $\C''$, with $\C$ finite, exact sequences $(\E)$  are classified
by \kt linear normal faithful Hopf monads $T$ on the tensor category $\C''$ whose induced Hopf algebra $H$ is such that $\CoRep{H}$ is tensor equivalent to $\C'$.

Equivariantization provides examples of  exact sequences of tensor categories. Indeed, if a finite group $G$ acts on a tensor
category $\C$ by tensor autoequivalences, then the
equivariantization $\C^G$ is again a tensor category, and we have an exact sequence of tensor categories:
$$\rep G \to \C^G \to \C.$$
An action $\rho$ of a finite group $G$ on a tensor
category $\C$ by tensor autoequivalences can be encoded in the form of a Hopf monad $T^\rho$ on $\C$, defined by $$T^\rho= \bigoplus_{g \in G} \rho(g)$$
as an endofunctor of $\C$,
so that  $\C^G$ is the category of $T^\rho$-modules
$\C^{T^\rho}$.

We show that a Hopf monad $T$ on a tensor category $\C$ is of the form $T^\rho$ for some group action $\rho$ if and only if $T$ is
\kt linear right exact faithful normal cocommutative (see Definition~\ref{def-cocom}) and its induced Hopf algebra is split semisimple.

A special case of an equivariantization is given by the
modularization procedure \cite{bruguieres, muger}. A premodular
category $\C$ over an algebraically closed field $\kk$ of
characteristic zero is modularizable if there exists a dominant
ribbon tensor functor  $F :\C \to  \widetilde{\C}$, where
$\widetilde{\C}$ is a modular category. Such is the case
if and only if the tensor subcategory of transparent objects $\T$
of $\C$ is tannakian, that is $\T \simeq  \rep G$ as a symmetric
tensor category, for some finite group $G$
(see~\cite{bruguieres}). In that case, we have an exact sequence
of fusion categories
$$\rep G \to \C \overset{F}\to \widetilde{\C},$$ where $F$ is the modularization functor, and $G$ acts on $\widetilde{\C}$ by braided tensor equivalences, so that $F$ is an equivariantization. Conversely, given a modular category $\D$, we classify premodular categories admitting $\D$ as a modularization in terms of \kt linear semisimple faithful  normal Hopf monads on $\D$ which are compatible with the ribbon structure.

We also interpret exact sequences of tensor categories in terms of commutative central algebras using  results of  \cite{blv}.
If $\C$ is a tensor category and $(A,\sigma)$ is a commutative algebra in the categorical center $\Z(\C)$ of $\C$, then the \kt linear abelian  category $\MOD{\C}{(A,\sigma)}$ of right $A$-modules in $\C$ admits a monoidal structure involving the half-braiding $\sigma$, so that the free module functor $F_A : \C \to \MOD{\C}{A}$,  $X \mapsto X \otimes A$ is strong monoidal.
We show that, given a dominant tensor functor $F : \C \to \D$ admitting an exact  right adjoint $R$, the object $A=R(\1)$ admits a canonical structure  of commutative algebra in the center of $\C$ such that $\MOD{\C}{(A,\sigma)}$ is a tensor category, and $F=F_A$ up to a tensor equivalence  $\D \simeq \MOD{\C}{A}$.
The central algebra $(A,\sigma)$ is called the \emph{induced central algebra of $F$.}
We show that $F$ is normal if and only if $A$ is \emph{self-trivializing,} that is, $F_A(A)$ is trivial. Then the induced Hopf algebra of $F$ is $\Hom(\1,F(A))$.

Thus, an exact sequence of tensor categories $\C' \toto \C
\overset{F}\toto \C''$ such that $F$ has an exact right adjoint is
equivalent to $\langle A\rangle \toto \C \toto
\MOD{\C}{(A,\sigma)}$, where $(A,\sigma)$ is the induced central
algebra of $F$  and $\langle A\rangle$ denotes the smallest abelian
subcategory of  $\C$ containing $A$ and stable by direct sums,
subobjects and quotients. Moreover, we show that $F$ is an
equivariantization if and only if $F(\sigma)$ is trivial and the
induced Hopf algebra of $F$ is split semisimple.

We introduce the notions of \emph{simple} fusion category and
\emph{normal} fusion subcategory arising naturally from the
definition of an exact sequence. If $\C$ is a fusion category and $\C' \subset \C$ is a fusion subcategory, we say that $\C' \subset \C$ is normal if it fits in an exact sequence of fusion categories $\C' \to \C \to \C''$. We say that $\C$ is simple if it has no non-trivial normal fusion strict subcategory.
 We characterize normal fusion subcategories in terms of commutative central algebras, and show that our notion of simplicity differs from that introduced in
\cite{eno2}. If $G$ is a finite
group, then the simplicity of $\rep G$ is equivalent to the
simplicity of $G$ and also to the simplicity of the fusion
category $\C(G)$ of $G$-graded vector spaces.

As an application of the notion of exact sequence of fusion
categories, we show the following classification
result:

\begin{theorem}\label{classif} Let $\mathcal C$ be a braided fusion category over an algebraically closed field $\kk$ of characteristic $0$.
If  $\dim \mathcal C$ is odd and square-free, then $\mathcal
C $ is equivalent to $\rep \Gamma$ as fusion categories for some finite group
$\Gamma$.
\end{theorem}
The proof relies on modularization and on the fact that a
quasitriangular Hopf algebra whose dimension is odd and
square-free is in fact a group algebra (\cite{qt-quotient}).

The paper is organized as follows: general definitions and  classical results are recalled in Section \ref{nuts}, which contains also elementary facts about Hopf monads. In Section \ref{two} we define dominant and normal tensor functors, and
exact sequences of tensor categories. We prove several fundamental results and study the relations between strictly exact sequences of Hopf algebras as defined by Schneider and exact sequences of tensor categories. In Section \ref{sect-es-fus}, we study exact sequences of fusion categories. Section \ref{sect-hm-es} is devoted to the classification of exact sequences of tensor categories in terms of Hopf monads, as well as equivariantization and the special case of modularization. In Section~\ref{sect-es-cca}, we revisit tensor functors and exact sequences of tensor categories in terms of commutative central algebras, and
study normal fusion subcategories and simple fusion categories.
Section
\ref{ribbon-sqfree} is devoted to the proof of Theorem
\ref{classif}.

\section{Tensor categories  and Hopf monads}\label{nuts} %

\subsection{Conventions and notations}
Monoidal categories will be strict, unless otherwise specified,
and the unit object will be denoted by $\1$.  A monoidal category
is \emph{rigid} (or \emph{autonomous}) if any object admits a left
dual and a right dual. If such is the case, the left dual and
right dual functors are denoted $\ldual{?}$ and $\rdual{?}$
respectively.

Most of the time, we work over a  base field $\kk$.
If $\A$ is an abelian \kt linear category, we say that an object
$X$ of $\A$ is \emph{scalar} if $\End(X)=\kk\id_X$. The category
$\A$ is \emph{split semisimple} if there is a set $\Lambda$ of
scalar objects in $\A$ such that every object of $\A$ is a finite
direct sum of elements of $\Lambda$, and such that $\Hom_{\C}(X, Y) = 0$ for $X \neq Y$ in
$\Lambda$. Note that
if $\A$ is split semisimple, scalar objects and simple objects
coincide in $\A$.


An abelian \kt linear category $\A$ is \emph{finite} if it is \kt
linearly equivalent to the category of finite dimensional left
modules over a finite dimensional \kt algebra. A \kt linear
functor $\ff : \A \to \B$ between finite abelian \kt linear
categories has a left (resp.\@ right) adjoint if and only if it is
left exact (resp.\@ right exact).


\subsection{Tensor categories and tensor functors}
A \emph{tensor category over $\kk$} is a \kt linear abelian rigid
monoidal category where Hom spaces are finite dimensional, all
objects have finite length, the monoidal product is \kt linear in
each variable, and the unit object $\1$ is scalar. The monoidal
product is then exact in each variable.

A tensor category over $\kk$ is \emph{finite} if it is finite as a
\kt linear abelian category.

A \emph{tensor functor} is
a \kt linear exact strong monoidal functor between two tensor
categories. Note that a tensor functor preserves duals and is
automatically  faithful.

If $H$ is a Hopf algebra over $\kk$, $\Rep{H}$  denotes the tensor
category of finite dimensional representations of $H$,  that is,
finite dimensional left $H$-modules. Similarly $\CoRep{H}$ denotes
the tensor category of finite-dimensional  right
$H$-comodules. In particular, if $G$ is a finite group, $\rep G :
= \Rep{\kk\,G} \simeq \CoRep{\kk^G}$ is the category of finite
dimensional representations of $G$, whereas $\CoRep{\kk \,G}$ is the
category of finite-dimensional $G$-graded vector spaces.

A morphism of Hopf algebras $f : H \to H'$ defines two tensor functors:
\begin{align*}
f_* :\prettydef{\CoRep{H} & \to & \CoRep{H'} \\ (M,\delta) &\mapsto&  (M,(\id_M \otimes f)\delta),} \\
f^* :\prettydef{\Rep{H'} & \to & \Rep{H} \\ (M,r) &\mapsto& (M, r(f \otimes \id_M)).}
\end{align*}

A \emph{fiber functor} for a tensor category $\C$ over $\kk$ is a
tensor functor $\omega: \C \to \vect_\kk$. By Tannaka theory,
given a  fiber functor $\omega$ for a tensor category $\C$ over
$\kk$, the coend $L(\omega)=\int^{X \in \C} \omega(X) \otimes_\kk
\rdual{\omega(X)}$ is a Hopf algebra over $\kk$, and we have a
canonical equivalence of tensor categories $\C
\overset{\simeq}\longrightarrow \CoRep{L(\omega)}$.

Let $\C$ be a tensor category. For a finite dimensional vector
space $E$ and an object $X \in \C$, their tensor product, denoted
by $E \otimes X \in \C$ is defined by the adjunction $$\Hom_{\C}(E
\otimes X, Y) \simeq \Hom_\kk(E, \Hom_{\C}(X, Y)).$$ The
assignment $(E,X) \mapsto E \otimes X$ makes $\C$ a left
$\vect_\kk$-module category. In particular, the functor $\vect_k
\to \C$, $E \mapsto E \otimes \1$ is a tensor functor from
$\vect_\kk$ to $\C$, and in fact the only such functor up to
tensor isomorphism. It is right adjoint to the \emph{global
section functor} $\Gamma: \C \to \vect_\kk$ defined by
$\Gamma(X)=\Hom_\kk(\1,X)$.

If $X$ is an object or set of objects of a tensor category $\C$,
we denote by $\langle X\rangle$ the smallest abelian subcategory
of $\C$ containing $X$ and stable by direct sums, subobjects and
quotients.

An object of a tensor category $\C$ is \emph{trivial} if it is
isomorphic to $\1^n$ for some natural integer $n$. The full subcategory of trivial objects of $\C$ is $\langle \1 \rangle \subset \C$.
It is a tensor category equivalent to $\vect_\kk$ via the tensor functor $X \mapsto \Hom_\C(\1,X)$.
A tensor category $\C$ is \emph{trivial} if $\C=\langle\1\rangle$, that is, if $\C$ is tensor equivalent to $\vect_\kk$.

\subsection{Existence of adjoints}\label{rladj}
Let $F: \C \to \D$ be a strong monoidal functor between rigid
categories. According to~\cite[Lemma 3.4]{bv-double}, if $F$ has a left adjoint $G$, then it has a right
adjoint $R$, and conversely.
In that case, $R$ and $G$ are related thus:
$$R(X) \simeq \ldual{G(\rdual{X})} \simeq \rdual{G(\ldual{X})} \blabla{and} R(X) \simeq\rdual{G(\ldual{X})} \simeq \ldual{G(\rdual{X})} \quad\mbox{(for $X$ in $\C$).}$$
In that case we say that \emph{$F$ has adjoints.}
A tensor functor between finite tensor categories has adjoints. In general, a tensor functor need not have adjoints; for instance, a fiber functor for a tensor category $\C$ has adjoints if and only if $\C$ is finite.

However, a tensor functor  $\ff : \C \to \D$  has an
Ind-adjoint and a Pro-adjoint because it is exact
(see~\cite{SGA4}). In other words the functor $\Ind \ff : \Ind \C
\to \Ind \D$ obtained by extending $\ff$ to the categories of
Ind-objects of $\C$ and $\D$ has a right adjoint $R : \Ind  \D
 \to \Ind \C$ called the Ind-adjoint of $\ff$, and, dually,
the functor $\Pro \ff : \Pro \C \to \Pro \D$ obtained by extending
$\ff$ to the categories of Pro-objects of $\C$ and $\D$ has a left
adjoint $G : \Pro \D \to \Pro \C$ called the Pro-adjoint of $\ff$.

\subsection{Fusion categories.}
A \emph{fusion category} over $\kk$ is a split semisimple finite
tensor category over $\kk$. Note that if $\kk$ is algebraically
closed, a tensor category $\C$ is split semisimple if and only if
it is semisimple. See for instance \cite[Section 2]{TY}.

Let $\C$ be a fusion category over $\kk$. The finite set of
isomorphism classes of simple (or scalar) objects in $\C$ will be
denoted by $\Lambda_\C$. The class of an object  $X$ of $\C$ in
the Grothendieck ring $K_0(\C)$ will be denoted by $[X]$. If $X\in
\Lambda_\C$ and  $Y$ is an object of $\C$, denote by $m_X(Y)$  the
multiplicity of $X$ in $Y$, that is: $m_X(Y) = \dim \Hom_\C(X,
Y)=\dim \Hom_\C(Y,X)$, so that we have:
$$Y \simeq \bigoplus_{X \in \Lambda_\C} X^{m_X(Y)}.$$
The \emph{Frobenius-Perron dimension} $\FPdim X$ of $X \in
\Lambda_{\C}$ is the largest positive eigenvalue of the matrix of
left multiplication by $X$ in the Grothendieck ring of $\C$ with
respect to the basis $\Lambda_{\C}$. It is  a real non-negative
algebraic integer. The \emph{Frobenius-Perron dimension} of $\C$
is  $\FPdim \C : = \sum_{X \in \Lambda_\C} (\FPdim X)^2$.

See \cite{ENO, muger-luminy} for a survey and a general reference
on fusion categories.

\subsection{Monads}
A \emph{monad} on a category $\A$ is an algebra $T$ in
the monoidal category $\End(\A)$ of endofunctors of $\A$. In other
words, it is an endofunctor $T$ of $\A$  endowed with natural
transformations $\mu: T^2 \to T$ and $\eta: \id_{\A} \to T$ (the
multiplication and unit of $T$, respectively), satisfying the
associativity and unit axioms in $\End(\A)$.

Let $T$ be a monad on $\A$. A \emph{$T$-module in $\A$} (often
called a $T$-algebra) is a pair $(M, r)$, where $M\in \D$ and $r :
T(M) \to M$ is a morphism in $\D$, such that
\begin{equation}\label{onr}rT(r) = r\mu_M, \qquad r\eta_M =
\id_M.\end{equation} A morphism $f: (M', r') \to (M, r)$ between
$T$-modules $(M', r')$ and $(M, r)$ is a morphism $f: M' \to M$ in
$\D$ such that $fr' = rT(f)$. This defines a category $\A^{T}$ of
$T$-modules in $\A$. Let $\U: \A^{T} \to \A$ denote the forgetful
functor: $\U(M, r) = M$. Then $\U$ admits a left adjoint $\Ll: \A
\to \A^{T}$, defined by $\Ll(X) = (T(X), \mu_X)$. We have $T =
\U\Ll$.

If $T$ and $T'$ are monads on $\C$, a \emph{morphism of monads $f
: T' \to T$} is a natural transformation such that $f\mu'_X =
\mu_Xf_{T(X)}T'(f_X)$ and $f_X\eta'_X = \eta_X$, for all object
$X$ of $\C$.

We will require the following characterization of faithful monads:

\begin{lemma} \label{lem-mon-faith}
Let $T$ be a monad on a category $\A$.
The following assertions are equivalent:
\begin{enumerate}[(i)]
\item For any $X$ object of $\A$, there exists a $T$\ti module $(M,r)$ and a monomorphism $X \to M$;
\item The unit $\eta$ of $T$ is monomorphic;
\item The monad $T$ is faithful;
\item The free module functor $\Ll : \A \to \A^T$ is faithful.
\end{enumerate}
\end{lemma}

\begin{proof}
Recall that $\Ll$ is left adjoint to the forgetful functor $\U : \A^T \to \A$, and $T=\U \Ll$.
Clearly we have (ii) $\Rightarrow$ (i). Conversely (i)
$\Rightarrow$ (ii): if $f : X \to M$ is a monomorphism, $(M,r)$ being a $T$\ti module, then
$\eta_X$ is a monomorphism because we have $f=r \eta_M f =r T(f)\eta_X$. Now, since $T=\U\Ll$ and $\U$ is faithful, we have
(iv) $\iff$ (iii). Also (iii) $\Rightarrow$ (ii): let $u,v  : Y \to
X$  be  morphisms such that  $\eta_X u=\eta_X v$. Then $T(u)= \mu_X T(\eta_X) T(u)=\mu_X T(\eta_X) T(v)=T(v)$ hence $u=v$ if $T$ is faithful.  Lastly, (ii) $\Rightarrow$ (iii).
Indeed let $u, v: X \to Y$ be  morphisms such that $T(u)=T(v)$. We have
$\eta_Y u=T(u)\eta_X=T(v) \eta_X= \eta_Y v$, so $u=v$ if $\eta_Y$ is a monomorphism.
\end{proof}

\subsection{Hopf monads}
Let $\C$ be a monoidal category. A monad on $\C$ is a
\emph{bimonad} when the category $\C^T$ is endowed with a monoidal
structure  such that the forgetful functor $\U : \C^T \to \C$ is
monoidal strict. This is equivalent to saying that the monad $T$
is endowed  with a structure of comonoidal endofunctor, that is, a
natural transformation
$$T_2(X,Y): T(X \otimes Y) \to T(X) \otimes T(Y)\quad \mbox{($X,Y$ in $\C$)}$$ and a morphism $T_0: T(\1) \to \1$ satisfying:
\begin{align*}
&(T_2(X,Y)\otimes \id_{T(Z)})T_2(X \otimes Y, Z)=(\id_{T(X)}\otimes T_2(Y,Z))T_2(X, Y \otimes Z),\\
&(\id_{T(X)} \otimes T_0)T_2(X,\1)=\id_{T(X)}=(T_0\otimes
\id_{T(X)} )T_2(\1,X);
\end{align*}
and  such that the product $\mu$ and the unit $\eta$ are monoidal
transformations, that is:
\begin{align*}
&T_2(X,Y)\mu_{X\otimes Y}=(\mu_X \otimes \mu_Y) T_2(T(X),T(Y))T(T_2(X,Y)),\\
& T_0 \mu_\1= T_0T(T_0), \quad T_2(X,Y) \eta_{X\otimes Y}=\eta_X
\otimes \eta_Y,\quad T_0 \eta_\1=\id_\1.
\end{align*}
Bimonads are introduced in \cite{moerdijk} under the name Hopf
monads.

If $\C$ is rigid, a bimonad $T$ on $\C$ is a \emph{Hopf monad} if
$\C^T$ is rigid; this is equivalent to saying that $T$ has a left
and a right antipode (see \cite[3.3]{bv}.) Hopf monads on
arbitrary monoidal categories are defined in \cite{blv}.

A \emph{morphism of bimonads} or \emph{Hopf monads} is a
comonoidal morphism of monads between  bimonads or Hopf monads.

\medbreak


\subsection{Monadicity}\label{comonmonadic}
Let  $(G: \A \to \B, \ff: \B \to \A)$ be an adjunction, with unit
$\eta: \id_{\A} \to \ff G$ and counit $\epsilon: G\ff \to
\id_{\B}$. Then $T = \ff G$ is a monad on $\A$, and there exists a
unique functor $\kappa: \B \to \A^{T}$ such that $\U\kappa = \ff$
and $\kappa G = \Ll$. The functor $\kappa$, called the
\emph{comparison functor} of the adjunction, is given by
$\kappa(X) = (\ff(X), \ff(\epsilon_X))$.

The adjunction $G \vdash \ff$ is called \emph{monadic} if the
comparison functor $\kappa$ is an equivalence. Necessary and
sufficient conditions for an adjunction to be monadic are given by
Beck's Theorem, see \cite[VI.7]{maclane}. In particular if $\A$
and $\B$ are abelian and the functor $\ff$ is additive and
faithful exact, then $G\vdash \ff$ is monadic.

Now let $\C$, $\D$ be monoidal categories and $\ff: \D \to \C$ be
a strong monoidal functor. Assume that $\ff$ has a left adjoint
$G$. Then \cite[Theorem 9.1]{bv} asserts that the functor $G$ is
comonoidal, the monad $\ff G$ in $\D$ has a canonical structure of
a bimonad and the comparison functor $\kappa: \C \to \D^T$ is
strong monoidal. Moreover, $\U\kappa = \ff$ as monoidal functors
and $\kappa G = \Ll$ as comonoidal functors. Lastly, if $\C$ and
$\D$ are rigid, then $T$ is a Hopf monad.

\subsection{Hopf monads and tensor categories}

Let $F : \C \to \D$ be a tensor functor between tensor categories
over a field $\kk$. Assume that $F$ admits a left adjoint $G$
(which is then unique up to unique isomorphism). Then, $F$ being
faithful exact, the adjunction $G \vdash F$ is monadic. The monad
$T=FG$ of this adjunction is a Hopf monad on $\D$, which is called
the Hopf monad of $F$. It is $\kk$ linear right exact, and we have
$\C \simeq \D^T$ as tensor categories.

Note that if  $F$ is  a  tensor functor between finite tensor categories (such as fusion categories)
then $F$ admits a left adjoint and so, it is monadic.

\begin{proposition} Let $\C$ be a tensor category over a field $\kk$, and let $T$ be a \kt linear right exact Hopf monad on $\C$. Then $\C^T$ is a tensor category over $\kk$, and
the forgetful functor $\U : \C^T \to \C$ is a tensor functor.
\end{proposition}
\begin{proof}
It is a general fact that the category of modules over a \kt
linear right exact monad on an abelian category is \kt linear
abelian and the forgetful functor is \kt linear exact. Applying
this to $T$, the category $\C^T$ is  \kt linear abelian and rigid;
its tensor product is \kt linear, and its unit object $(\1,T_0)$
is scalar. Moreover, the forgetful functor $\U$ is monoidal strict
and \kt linear faithful exact; in particular in $\C^T$, $\Hom$'s
are finite dimensional and objects have finite length. Thus,
$\C^T$ is a tensor category over $\kk$, and $\U$ is a tensor
functor.
\end{proof}

\begin{example}\label{hopf-alg} A Hopf algebra $H$ in a braided autonomous category $\B$ defines
a Hopf monad $H \otimes ?$ on $\B$, see \cite[Example 3.10]{bv} and \cite[Example 2.8]{bv-double}.
In particular a finite dimensional Hopf algebra $H$ over $\kk$
defines a Hopf monad $H \otimes ?$ on the category $\vect_\kk$ of
finite dimensional vector spaces. It is the monad of the  forgetful functor
$$\Rep H \to \vect_\kk.$$
\end{example}

In fact, \kt linear Hopf monads on trivial tensor categories are just finite dimensional Hopf algebras:

\begin{lemma}\label{HM-triv}
Let $\C$ be a trivial tensor category. If $H$ is a finite
dimensional Hopf algebra over $\kk$, then $H \otimes ? : \C \to
\C$ admits a natural structure of \kt linear Hopf monad on $\C$.
The assignment $H \mapsto H \otimes ?$ defines an equivalence of
categories between finite dimensional Hopf algebras over $\kk$ and
\kt linear Hopf monads on $\C$.
\end{lemma}

\begin{proof} A trivial tensor category, being by definition
tensor equivalent to $\vect_\kk$, admits a unique braiding and is symmetric.
The tensor functor $\vect_k \to \C$, $E \mapsto E \otimes \1,$ is
symmetric and sends a finite dimensional Hopf algebra $H$ over
$\kk$ to a Hopf algebra $H \otimes \1$ in $\C$. Thus $H \otimes ?:
\C \to \C$ is \kt linear Hopf monad on $\C$. Now if $T$ is a \kt
linear Hopf monad on $\C$, set $H=\Hom_\C(\1,T(\1))$. We have a
canonical isomorphism $a_\1 : H \otimes \1 \iso T(\1)$, which
extends uniquely to a natural isomorphism $a: H \otimes ? \iso T$
because $\C$ is semisimple and $\Lambda_\C=\{\1\}$. One verifies
that there exists a unique structure of Hopf algebra on $H$ such
that, when $H\otimes ?$ is endowed with the corresponding
structure of Hopf monad, $a$ becomes an isomorphism of Hopf monads
on $\C$.
\end{proof}

\section{Exact sequences of tensor categories}\label{two}
In this section, we introduce the notions of normal tensor functor and of exact sequences of tensor
categories over a field $\kk$.

\subsection{Dominant functors, normal functors and exact sequences}

\begin{lemma}\label{dom-sub-quot}
Let  $\ff : \C \to \D$ be a tensor functor between tensor categories. The following assertions are equivalent:
 \begin{enumerate}[(i)]
 \item Any object  $Y$ of $\D$ is a subobject of $\ff(X)$ for some object $X$ of $\C$;
 \item  Any object  $Y$ of $\D$ is a quotient of $\ff(X)$ for some object $X$ of $\C$;
 \item The Pro-adjoint of $\ff$ is faithful;
 \item The Ind-adjoint of $\ff$ is faithful.
\end{enumerate}
\end{lemma}

\begin{proof} Assume (i) and let $Y$ be an object of $\D$. There exists $X$ object of $\C$ and a monomorphism $i : \rdual{Y} \to \ff(X)$. In a rigid category, the left dual and right dual functors are quasi-inverse contravariant equivalences, and strong monoidal functors preserve duals.
Hence $\ldual{i} : \ff(\ldual{X})\simeq \ldual{\ff(X)} \to \ldual{\rdual{Y}} \simeq Y$ is an epimorphism, hence (i) \implique (ii); the converse is proved similarly. Hence (i) \ssi (ii).

Now let us prove (i) \ssi (iii). The Pro-adjoint functor of $\ff$ is the left adjoint functor $G$ of the functor $\Pro f : \Pro \C \to \Pro D$. Now $\Pro \C$ and $\Pro \D$ are abelian categories, and $\Pro f$ is \kt linear faithful exact. By Beck's theorem, the adjunction $G \vdash \Pro F$ is monadic, with monad $T= \Pro F \, G$, so
$\Pro\C$ is equivalent to  $(\Pro \D)^T$ via the comparison functor, $F$ being the forgetful functor. Denote by $\eta$ the unit of the monad $T$.
Assume (i) holds. Then for any object $Y$ in $\Pro \D$, $\eta_Y$ is a monomorphism. Indeed if $Y$ is in $\D \subset \Pro \D$, then there exists $X$ in $\C$ and a monomorphism $i: Y \to F(X)$. Since $\C$ is a full subcategory of $\Pro \C \simeq (\Pro \D)^T$, we may view $X$ as a $T$-module, with action $r : T F(X) \to F(X)$, and $i= r T(i)\eta_Y$ so $\eta_Y$ is a monomorphism in that case. In general an object $Y$ of $\Pro \D$ is of the form $\Lim Y_i$ for some filtering system $(Y_i)$ of objects of $\D$, and $\eta_{Y_i}$ is a monomorphism for all $i$, hence $\eta_Y$ is a monomorphism. By Lemma~\ref{lem-mon-faith}, this implies that $G$ is faithful. Conversely if $G$ is faithful, again by Lemma~\ref{lem-mon-faith} for any object $Y$ of $\Pro \D$ there exists a filtering system $(X_i)$ of objects of $\C$ such that $Y$ is a subobject of $\Pro F(\Lim X_i)=\Lim F(X_i)$.
If $Y$ is in $\D$, it has finite length so there exists $i$ such that $Y \to F(X_i)$ is a monomorphism. Hence (i) \ssi (iii).

Lastly, (ii) \ssi (iv) results from (i) \ssi (iii) applied to the opposite functor $\ff^\opp : \C^\opp \to \D^\opp$, hence the Lemma is proved.
\end{proof}

\begin{definition}\label{def-dominantfunctor} A tensor functor $\ff : \C \to \C''$  is \emph{dominant} if it satisfies the equivalent conditions of Lemma~\ref{dom-sub-quot}.
\end{definition}

\begin{remark}
A dominant tensor functor in surjective in the sense of \cite{EO}.
\end{remark}

Let $\ff: \C \to \D$ be a tensor functor between tensor
categories. We denote by $\KER_\ff$ the full tensor subcategory
$\ff^{-1}(\langle \1 \rangle) \subseteq \C$  of objects $X$ of
$\C$ such that $\ff(X)$ is a trivial object of $\C$.

\begin{definition}\label{def-normalfunctor}
Let $\ff : \C \to \D$ be a tensor functor between tensor
categories. Then $\ff$ is \emph{normal} if for any object $X$ of
$\C$, there exists a subobject $X_0 \subset X$ such that
$\ff(X_0)$ is the largest trivial subobject of $\ff(X)$.
\end{definition}

%

\begin{proposition}\label{func-normal} Let $\ff : \C \to \D$ be a tensor functor between tensor categories.
\begin{enumerate}
\item If $\ff$ admits  a right adjoint $R$, or, equivalently, a left adjoint $G$,  then $\ff$ is normal if and only if $G(\1)$ belongs to $\KER_\ff$, if and only if $R(\1)$ belongs to $\KER_\ff$.
\item If $\C$ and $\D$ are fusion categories, $\ff$  is normal if and only if any  simple object $X$ of $\C$ such that
$\Hom(\1,\ff(X)) \neq 0$ belongs to $\KER_\ff$.
\end{enumerate}
\end{proposition}

\begin{proof}
Let us prove Part (1). For $X$ in $\C$, denote by $X_0 \subset X$
the largest subobject of $X$ belonging to $\KER_\ff$, which exists
because objects have finite length in $\C$. We have a commutative
diagram
$$\xymatrix{
\Hom_\D(\1,\ff(X_0)) \ar[d]_\sim\ar[r]^{a_X}& \Hom_\D(\1,\ff(X))\ar[d]^\sim\\
\Hom_\C(G(\1),X_0) \ar[r]_{b_X} & \Hom_\C(G(\1),X), }
$$
where the vertical arrows are the adjunction isomorphisms, and the
horizontal arrows $a$ and $b$ are induced by the inclusion $X_0
\subset X$. If $G(\1)$ belongs to $\KER_\ff$, then for all $X$,
$b_X$ is bijective, so $a_X$ is bijective, which means that $F(X_0)$ is the largest trivial subobject of $F(X)$. Hence $\ff$
is normal.
Conversely, suppose $\ff$ is normal, and let $X_0 \subseteq X$ be such that $\ff(X_0)$ is the largest trivial subobject of $F(X)$. Thus  $a_X$ is bijective, so
$b_X$ is bijective, for all object $X$ of $\C$. In particular $b_{G(\1)}$ is bijective,  so
$G(\1)_0=G(\1)$, hence $G(\1)$ belongs to $\KER_\ff$. Thus $\ff$
is normal $\iff$ $G(\1)$ belongs to  $\KER_\ff$ $\iff$
$R(\1)=\rdual{G(\1)}$ belongs to $\KER_\ff$.

Let us now prove Part (2). If $\C$ and $\D$ are fusion categories,
$\ff$ admits a left adjoint $G$, and $\ff$ is normal if and only
if $G(\1)$ is in $\KER_\ff$ by Part (1). By adjunction, we have
$m_X G(Y)= m_Y \ff(X)$ for all $X \in \Lambda_\C$ and $Y \in
\Lambda_\D$. In particular, $m_\1 \ff(X) > 0$ if and only if $m_X
G(\1) > 0$, hence Part (2).
\end{proof}

%
%
%
%
%

\begin{lemma}\label{lem-toto}
Let $\ff : \C \to \D$ be a tensor functor between tensor categories. Then:
\begin{enumerate}
\item The functor $\ff$ is an equivalence if and only if it is full and dominant.
\item The functor $\ff$ is full if and only if $\ff$ is normal and $\KER_\ff$ is trivial;
\end{enumerate}
\end{lemma}

\begin{proof}
Part (1): clearly if $\ff$ is an equivalence it is both full and dominant. Conversely, assume $\ff$ is full and dominant. Let $Y$
be an object of $\D$. By Lemma~\ref{dom-sub-quot}, there exist $X_1$, $X_2$ in $\C$, an epimorphism $p:\ff(X_1) \to Y$ and a monomorphism $i:Y \to \ff(X_2)$, and $Y$ is (isomorphic to) the image of $ip$. Since $\ff$ is full, there exists $\pi : X_1 \to X_2$ such that $\ff(\pi)=ip$. Let $X$ be the image of $\pi$; since $\ff$ is exact, it preserves images so $\ff(X) \simeq Y$. Thus $\ff$ is essentially surjective, and is therefore an equivalence.

Part (2): if $\ff$ is full, it is normal and $\KER_\ff$ is trivial. Conversely, assume $\ff$ is normal and $\KER_\ff$ is trivial. Then for any $X$ in $\C$ we have a subobject $X_0 \subset X$ such that $\ff(X_0)$ is the largest trivial subobject of $\ff(X)$. In particular $X_0$ is trivial,
so $\ff$ induces an
isomorphism $\Hom_\C(\1, X) \to  \Hom_{\D} (\1,\ff(X))$. Since in a rigid category
$\Hom(X,X') \simeq \Hom(\1, \rdual{X} \otimes X')$, we conclude that  $\ff$ is fully faithful.
\end{proof}

\begin{definition}\label{def-sec} Let $\C', \C, \C''$ be tensor categories over $\kk$. A sequence of tensor functors
\begin{equation}\label{exacta-fusion}\xymatrix{\C' \ar[r]^f & \C \ar[r]^\ff & \C''}
\end{equation}
is an \emph{exact sequence of tensor categories} if the
following conditions hold:
\begin{enumerate}[(1)]
\item  The tensor functor $\ff$ is dominant and normal;
\item The tensor functor $f$ is a full embedding;
\item The essential image of $f$ is $\KER_\ff$;
\end{enumerate}

Two exact sequences of  tensor  categories
$$\xymatrix{\C' \ar[r]^{f_1} & \C_1 \ar[r]^{\ff_1} & \C''} \quad\mbox{and}\quad \xymatrix{\C' \ar[r]^{f_2} & \C_2 \ar[r]^{\ff_2} & \C''}$$ are \emph{equivalent} if there exists
a  tensor equivalence  $\lambda: \C_1 \to \C_2$ such
that $\ff_1 \simeq \ff_2 \lambda$ and $f_2 \simeq \lambda f_1$ as
tensor functors,  that is, such that the diagram  
$$\begin{CD}\C' @>{f_1}>> \C_1 @>{\ff_1}>> \C'' \\
@VV{=}V @VV{\lambda}V @VV{=}V\\
\C' @>{f_2}>> \C_2 @>{\ff_2}>> \C''.
\end{CD}$$
is commutative up to  monoidal  isomorphism.

 A exact sequence $\C' \to  \C \to\C''$ is called an \emph{extension of $\C''$ by $\C'$}; we also say that $\C$ is an extension of $\C''$ by $\C'$.
\end{definition}


Note that a normal dominant tensor functor $\ff: \C \to \C''$
between tensor categories defines an exact sequence of tensor
categories $$\KER_\ff \to \C \overset{\ff}\to \C''.$$

\begin{proposition}\label{trivial-sec} If $\C' \overset{f}\to \C \overset{\ff}\to \C''$
is an exact sequence of tensor categories, then:
\begin{enumerate}

\item The tensor functor $\ff$ is an equivalence if and only
if $\C'$ is trivial;

\item The tensor functor $f$ is an equivalence if and only
if $\C''$ is trivial. \end{enumerate}\end{proposition}

\begin{proof} Part (1): if $\ff$ is an equivalence then $\KER_\ff=\langle\1\rangle$, that is, $\C'$ is trivial. Conversely, if $\KER_\ff$ is trivial then
by Lemma~\ref{lem-toto} $\ff$ is an equivalence.

Part (2): if $\C''$ is trivial, $\KER_\ff=\C$, hence $f$ is an
equivalence. Conversely if $f$ is an equivalence, $\KER_\ff=\C$.
Hence $\ff(\C) \subset \langle\1\rangle$, and, $\ff$ being
dominant, $\C''=\langle\1\rangle$, that is, $\C''$ is trivial.
\end{proof}

\subsection{Exact sequences of tensor categories from Hopf algebras}

Strictly exact sequences of Hopf algebras as defined in \cite{schn2} give rise to exact sequences of tensor categories.
We always assume that Hopf algebras have a bijective antipode. If $H$ is a Hopf algebra, we denote by $H^+ \subset H$
the \emph{augmentation ideal} $H^+=\{x \in H \mid \varepsilon(x)=0\}$.

In the category of Hopf algebras over a field $\kk$, the trivial Hopf algebra $\kk$ is a zero object, that is, it is both initial and final.
A morphism $f : H \to H'$ of Hopf algebras over $\kk$ admits a categorical kernel and a categorical cokernel, defined by
\begin{align*}
\ker(f)&=\{x \in H \mid x_{(1)} \otimes f(x_{(2)}) \otimes x_{(3)} = x_{(1)} \otimes 1 \otimes x_{(2)} \}\\
\intertext{in Sweedler's notation, and}
\coker(f)&=H'/H' f(H^+).
\end{align*}

Observe that $\ker(f)$ is not $f^{-1}(0)$ and $\coker(f)$ is not $H'/f(H)$.

A \emph{strictly exact sequence of Hopf algebras} is a diagram
$$K \overset{i}\toto  H \overset{p}\toto H'$$
where $i$, $p$ are morphisms of Hopf algebras  such that
\begin{enumerate}[(a)]
\item $K$ is a normal Hopf subalgebra of $H$;
\item  $H$ is right faithfully flat over $K$;
\item $p$ is a categorical cokernel of $i$,
\end{enumerate}
or, equivalently, setting $I=p^{-1}(0)$, such that
\begin{enumerate}[(a')]
\item $I$ is a normal Hopf ideal of $H$;
\item $H$ is right faithfully coflat over $H'$;
\item $i$ is a categorical kernel of $p$.
\end{enumerate}
 A Hopf subalgebra $K \subset H$ is normal if it is a submodule of $H$ for the adjoint action of $H$ on itself, defined by $x.y= x_{(1)}y S(x_{(2)})$, and a Hopf ideal $I \subset H$ is normal if it is a subcomodule of $H$ for the coadjoint coaction of $H$ on itself, defined by $X \mapsto x_{(2)} \otimes x_{(1)} S(x_{(3)})$.

\begin{proposition}\label{prop-schne}
A strictly exact sequence  $K \overset{i}\toto  H \overset{p}\toto H'$  Hopf algebras over a field gives rise to an exact sequence of tensor categories:
$$\CoRep{K} \overset{i_*}\toto \CoRep{H} \overset{p_*}\toto \CoRep{H'}$$ and also, if $H$ is finite-dimensional, to a second exact sequence of tensor categories:
$$\Rep{H'} \overset{p^*}\toto \Rep{H} \overset{i^*}\toto \Rep{K}.$$
\end{proposition}

\begin{proof}
Let $f : H \to H'$ be a morphism of Hopf algebras over $\kk$.
Denote respectively by  $H^{\co H'}$ and $\,^{\co H'}\!\!H$ the
subalgebras of $H$ of right and left $H'$ coinvariants, \emph{i.
e:} $$H^{\co H'} = \{ h \in H/\, (\id_H \otimes f)\Delta(h) = h
\otimes 1 \}, \quad  ^{\co H'}\!\!H = \{  h \in H/\, (f \otimes
\id_H)\Delta(h) = 1 \otimes h \}.$$

\begin{lemma}\label{lem-hopf-norm}
Let $f : H \to H'$ be a morphism of Hopf algebras over a field. Then
\begin{enumerate}\item The tensor functor $f_*:  \CoRep{H}\to \CoRep{H'}$ is normal if and only if
$H^{\co H'}=\,^{\co H'}\!\!H$. If such is the case, we have $\KER_{f_*}= \CoRep{\ker(f)}$, with $\ker(f)=H^{\co H'}$.
\item The tensor functor $f^* : \Rep{H'} \to \Rep{H}$ is normal if $f(H)$ is a normal Hopf subalgebra of $H'$, and in that case we have $\KER_{f^*}= \Rep{\coker(f)}$, with $\coker(f)={H'/H'f(H^+)}$.
\end{enumerate}
\end{lemma}

\begin{proof}
Part (1).
If $(M,\delta)$ is a right $H$-comodule, then $f_*(M,\delta)=(M,(\id_M \otimes f)\delta)$, and the
largest trivial subobject of $f_*(M,\delta)$ is $$M^{\co H'} = \{ x \in M/\, (\id_M \otimes f) \delta(x) = x\otimes
1 \}.$$
According to Definition \ref{def-normalfunctor}, $f_*$ is
normal  if and only if for all finite-dimensional right $H$-comodule $M$, $M^{\co H'} \subset M$ is a subcomodule.
This is equivalent to saying that for all right $H$-comodule $M$,
$M^{\co H'} \subset M$ is a subcomodule, because any comodule is locally finite.

Now assume $f_*$ is normal. Then $H^{\co H'}$ is a subcomodule of the right comodule $H=(H,\Delta)$. If $h \in H^{\co H'}$ we have
 in Sweedler's notation $\Delta(h)=h_{(1)} \otimes h_{(2)}$, with
$h_{(1)} \in H^{\co H'}$.  Thus $h_{(1)} \otimes f(h_{(2)}) \otimes h_{(3)}=h_{(1)} \otimes 1 \otimes h_{(2)}$, and so
$f(h_{(1)}) \otimes h_{(2)} = 1
\otimes h$, and $h \in ^{\co H'}\!\!H$.
Thus $H^{\co H'} \subset \,^{\co H'}\!\!H$. The reverse inclusion follows from the fact that the antipode of $H$,  being an anti-bialgebra isomorphism by assumption,  exchanges $H^{\co H'}$ and $^{\co H'}\!\!H$.
Hence $f$ is normal.

Conversely, assume $H^{\co H'}=\,^{\co H'}\!\!H$. Let $(M,\delta)$ be a right $H$-comodule.
We have $\delta(M^{\co H'}) \subseteq M \otimes H^{\co
H'}$. Indeed, for  $x \in M$ let $\delta(x) = x_{(0)} \otimes x_{(1)}$ in Sweedler's notation. If
 $x\in M^{\co H'}$, we have
$$x_{(0)} \otimes (\id_M \otimes f) \Delta(x_{(1)})  = \delta x_{(0)}
\otimes  f(x_{(1)}) = \delta(x) \otimes 1 = x_{(0)} \otimes x_{(1)}
\otimes 1.$$ Now by assumption $H^{\co H'}=^{\co H'}\!\!H$, so for $x \in M^{\co H'}$ we have
$$(\id_M \otimes f) \delta  (x_{(0)}) \otimes x_{(1)}=x_{(0)} \otimes (f \otimes \id_H)\Delta (x_{(1)})=x_{(0)} \otimes 1 \otimes x_{(1)},$$
so $M^{\co H'}$ is a subcomodule of $M$, hence $f_*$ is normal.

If $f$ is normal, $K = H^{\co  H'} = \,^{\co H'}\!\!H=\ker(f)$ is a
Hopf subalgebra of $H$ and $\CoRep{K}$ is the full tensor subcategory
of $\CoRep{H}$ whose  objects are  those right $H$-comodules which are trivial as  $H'$-comodule. In other words
$\CoRep{K} = \KER_{f_*}$.

Part (2). Let $(M,r)$ be a finite-dimensional left $H'$-module.
The largest trivial subobject of  $f^*(M,r)=(M,r(f \otimes
\id_M))$ is $$M_0=\{m \in M \mid \forall x \in H,
f(x)m=\varepsilon(x)m\}= \{m \in M \mid f(H^+)m=0\}.$$ If $f(H)$
is a normal Hopf subalgebra of $H'$, we have
$f(H^+)H'=H'f(H^+)$, hence  $M_0$ is a $H'$-submodule of $M$ and
so, $f^*$ is normal. Moreover
$\KER_{f^*}=\Rep{H'/H'f(H^+)}=\Rep{\coker(f)}$.
\end{proof}

%
%
%


\begin{lemma}\label{lem-hopf-dom}
Let $f : H \to H'$ be a morphism of Hopf algebras over a field. Then
\begin{enumerate}
\item The tensor functor $f_* : \CoRep{H} \to \CoRep{H'}$ is dominant if and only if $?\, \square^{H'}H :  \COREP{H'} \to \COREP{H}$ is faithful;
\item The tensor functor $f^* : \Rep{H'} \to \Rep{H}$ is dominant if $f$ is injective and $H'$ is finite-dimensional.
\end{enumerate}
\end{lemma}

\begin{remark}
In particular, if $f_*$ is dominant then $f$ is surjective. Conversely if $f$ is surjective and $H$ is right $H'$ coflat, then $f_*$ is dominant, with exact Ind-adjoint. If $H$ is finite-dimensional, it is right $H'$ coflat; and in that case, $f_*$ is dominant if and only if $f$ is surjective.
\end{remark}

\begin{proof} Part (1). If $C$ is a coalgebra over a field,
the category of Ind-objects of the category of finite-dimensional right $C$-comodules  is the category $\COREP{C}$ of all right $C$-comodules. The Ind-adjoint $R$ of $f_*$ is the right adjoint
of  $$\Ind(f_*) : \prettydef{\COREP{H} &\to &\COREP{H'}\\(X,\delta)&\mapsto &(X,(\id_X \otimes f)\delta),}$$ that is, $R=?\, \square^{H'}H$.
By Lemma~\ref{dom-sub-quot}, $f_*$ is dominant if and only if $R$ is faithful.

Part (2). If $H'$ is finite-dimensional and $f$ is injective, then
$H'$ is a free left $H$-module  \cite{NZ}. If $M$ is a
finite dimensional left $H'$-module, then $f \otimes_{H'} M :
M \simeq H' \otimes_{H'} M \to H \otimes_{H'} M$ is a monomorphism
$M \to f_* (H \otimes_{H'} M)$, hence $f_*$ is dominant.

\end{proof}

Now let us conclude the proof of Proposition~\ref{prop-schne}. Consider a  strictly exact sequence  $K \overset{i}\toto  H \overset{p}\toto H'$  of  Hopf algebras over a field $\kk$. By Lemma~\ref{lem-hopf-dom}, $p_*$ is dominant because $H$ is right faithfully coflat over $H'$.
Set $I=p^{-1}(0)$. The fact that $I$ is, by assumption, a normal Hopf ideal of $H$ means that the morphism $p : H \to H' \simeq H/I$ is conormal. By \cite[Lemma 1.3]{schn2} we have $H^{\co  H'} = \,^{\co H'}\!\!H$, so by Lemma~\ref{lem-hopf-norm} the tensor functor $p_*$ is normal, and $\KER_{p_*}=\CoRep{K}$ since $K=\ker(p)$. Hence $\CoRep{K} \to \CoRep{H} \to \CoRep{H'}$ is an exact sequence of tensor categories.
Now assume $H$ is finite-dimensional. Then by Lemma~\ref{lem-hopf-dom} the tensor functor $i^* : \Rep{H} \to \Rep{K}$ is dominant because $i$ is injective, and by Lemma~\ref{lem-hopf-norm} it is normal,  with $\KER_{i^*}=\Rep{H'}$ because $H'=\coker(i)$.
Hence $\Rep{H'} \to \Rep{H} \to \Rep{K}$ is an exact sequence of tensor categories.
\end{proof}

\begin{remark}
Let $\C' \overset{f}\to \C \overset{F}\to \C''$ be an exact sequence of tensor categories over a field $\kk$ such that the Ind-adjoint of $F$ is exact.
Assume moreover that $\C''$ admits a fiber functor $\omega : \C'' \to \vect_\kk$. Then, setting $H'=L(\omega)$, $H''=L(\omega F)$ and $K=L(\omega Ff)$,
and denoting by $i : K \to H$ and $p : H \to H'$ the Hopf algebra morphisms induced by $f$ and $F$ respectively, we obtain a strictly exact sequence of Hopf algebras over $\kk$:
$$K\overset{i}\toto H \overset{p} \toto H',$$
and we have an isomorphism of exact sequences of tensor categories
$$\xymatrix{
\C' \ar[d]_\simeq \ar[r]^{f} & \C \ar[d]^\simeq \ar[r]^{F} &\C'' \ar[d]^\simeq\\
\CoRep{K} \ar[r]_{i_*} & \CoRep{H} \ar[r]_{p_*} & \CoRep{H''}.
}$$
Note that we must assume that the Ind-adjoint of $F$ is exact because the definition of strictly exact sequence of tensor categories implies that $H$ is right coflat over $H'$.
\end{remark}

Applying the proposition to exact sequences of group algebras, we obtain:

\begin{corollary}\label{sec-groups} An exact sequence of finite groups $1 \to G''
\overset{\iota}\to G \overset{\pi}\to G' \to 1$ gives rise to an exact sequence of
tensor categories:
$$\xymatrix{\rep G' \ar[r]^{\pi*}& \rep G \ar[r]^{\iota^*}& \rep G''}.$$
\end{corollary}

\subsection{Induced Hopf algebras}\label{fiberf}

Let $\ff: \C \to \C''$ be a tensor functor between tensor
categories over a field $\kk$. The functor
$$\omega_\ff:\prettydef{ \KER_\ff &\to &\vect_\kk,\\ X &\mapsto
&\Hom(\1,\ff(X))}$$ is a fiber functor for $\KER_\ff$ because
$\ff(\KER_\ff) \subset \langle\1\rangle$. Then
$$L(\omega_\ff)=\int^{X \in \KER_\ff} \omega_\ff(X) \otimes_\kk
\rdual{\omega_\ff(X)}$$ is a Hopf algebra, and we have a canonical
tensor equivalence $$\KER_\ff \iso \CoRep{L(\omega_\ff)}.$$

If $\ff$  is a normal tensor functor, $\L(\omega_\ff)$ is called the \emph{induced Hopf algebra of $\ff$}.

If $\C' \overset{i}\to \C \overset{\ff}\to \C''$ is an exact
sequence of tensor categories over $\kk$, we have a canonical tensor
equivalence $\C' \iso \CoRep{L(\omega_\ff)}$. Then
$L(\omega_\ff)$ is called the \emph{induced Hopf algebra of the short exact sequence}.

\begin{proposition}\label{prop-crit-adj}
Let $\C' \overset{f}\toto \C \overset{F}\toto \C''$ be an exact sequence of tensor categories, with induced Hopf algebra $H$. The following assertions are equivalent:
\begin{enumerate}[(i)]
\item The functor $F$ has  adjoints;
\item The tensor category $\C'$ is finite;
\item The Hopf algebra $H$ is finite-dimensional.
\end{enumerate}
In particular, if $\C'$ and $\C''$ are finite, so is $\C$.
\end{proposition}

\begin{proof} We may assume $\C'=\KER_F$, $f$ being the inclusion.
We have (ii) \ssi (iii) because  $\KER_\ff \simeq \CoRep{H}$. If $F$ has a right adjoint $R$, then the fiber functor $\omega= \Hom(\1, Ff) : \C' \to \vect_\kk$
has a right adjoint $? \otimes R(\1)$, hence by adjunction $$H = \int^{X \in \C'}  \rdual{\omega(X)} \otimes \omega(X) \simeq \int^{E \in \vect_\kk}  \rdual{E}\otimes
E \otimes \omega R(\1)=
\Hom_{\C''}(\1,FR(\1))$$
is finite-dimensional. Hence (i) \implique (iii). Conversely, assume $H$ is finite-dimensional,  and denote by $R : \Ind \C'' \to \Ind \C$  the Ind-adjoint of $F$, that is, the right adjoint of $\Ind F$.
The functor $F$ has a right adjoint if and only if for any $X$ in $\C''$, $R(X)$ is isomorphic to an object of $\C$.
Now $\Ind F$ is strong monoidal, and in particular comonoidal, so its right adjoint  $R$ is monoidal, hence we have a natural transformation $R_2(X,Y) : R(X) \otimes R(Y) \to R(X\otimes Y)$.
Let $X$ be in $\C$ and let $$H_X = R_2(\1, F(X))(\id_{R(\1)} \otimes h_X)  : R(\1) \otimes X \to RF(X),$$ where $h$ denotes the evaluation of the
adjunction $\Ind F \vdash R$. We claim that $H_X$ is an isomorphism. In order to verify this, it is enough to check that for all $Y$ in $\C$, $H_X$ induces a bijection
$\Hom_\C(Y,R(\1) \otimes X) \simeq \Hom_\C(Y,RF(X))$.  We have $\Hom_\C(Y, R(\1) \otimes X) \simeq \Hom_\C(Y \otimes \rdual{X}, R(\1)) \simeq \Hom_{\C''}(F(Y \otimes \rdual{X}), \1)
\simeq \Hom_{\C''}(F(Y),F(X)) \simeq \Hom_\C(Y, RF X)$, and this bijection is the natural map induced by $H_X$. Hence $H_X$ is an isomorphism.
Let $Y$ be an object of $\C''$. Then, $F$ being dominant, $Y$ is a subobject of $F(X)$ for some $X$ in $\C$. Since $RF(X) \simeq
R(\1) \otimes X$, $RF(X)$ belongs to $\D$ because $R(\1) \simeq H \otimes \1$, and $R(X)$, being a subobject of $RF(X)$ because $R$ is left exact, is isomorphic to an object of $\D$,
so $F$ has a right adjoint; it also has a left adjoint by Remark~\ref{rladj}, hence (iii) \implique (i).
\end{proof}

\section{Exact sequences of fusion categories}\label{sect-es-fus}

Let $\ff : \C \to \D$ be a tensor functor between fusion categories. Then $\ff$ is dominant if and only if any simple object $Y$ of $\D$ is a direct factor of $\ff(X)$ for some simple object $X$ of $\C$, and $\ff$ is normal if and only if, for  any simple object $X$ of $\C$, if $\ff(X)$ contains a copy of the unit of $\1$ then $\ff(X)$ is trivial.

\begin{example} Let  $K \subset H$ be an inclusion of finite-dimensional split semisimple Hopf
algebras. 
Then the restriction functor from $\C=\Rep{H}$ to $\D=\Rep{K}$ is a dominant tensor functor between fusion categories. We have $\FPdim \D = \dim K$ and $\FPdim
\C = \dim H$, so that the quotient $\FPdim \C / \FPdim \D = \dim
H / \dim K$ is a natural integer, called the index of $K$ in $H$.
\end{example}

This example motivates the following definition.

\begin{definition}\label{fp-index} Let $\ff: \C \to
\D$ be a dominant tensor functor between fusion
categories. The \emph{Frobenius-Perron index} of $\ff$ is the
ratio $\FPdim \C / \FPdim \D$; we denote it  by $\FPind (\ff)$
or $\FPind(\C:\D)$.
\end{definition}

It follows from \cite[Corollary 8.11]{ENO} that if $\ff: \C \to
\D$ is a dominant tensor functor, then $\FPdim \D$ divides
$\FPdim \C$, that is, the Frobenius-Perron index of a dominant
tensor functor $\ff$ between fusion categories is always an algebraic integer. More precisely:

\begin{proposition}\label{inddom}
If $\ff: \C \to \D$ is a dominant tensor functor between fusion
categories, then $$\FPind(\ff)= \FPdim G(\1),$$ where $G$ is a left
(or a right) adjoint of $\ff$.
\end{proposition}

\begin{proof}

For any fusion category $\C$, let  $R(\C) = K_0(\C)
\otimes_{\mathbb Z}\mathbb R$ denote the $\mathbb R$-algebra
obtained by extension of scalars from the Grothendieck ring
$K_0(\C)$ of $\C$. Consider the element $$R_\C : = \sum_{X \in
\Lambda_{\C}} \FPdim X \, [X] \in R(\C).$$ We have $\FPdim R_\C =
\FPdim \C$.

\begin{lemma}\label{FPdimdom}\cite{ENO}
Let $\ff : \C \to \D$ be a dominant tensor functor between fusion
categories. Then
\begin{equation}\label{fdereg}\ff_!(R_\C) = \frac{\FPdim \C}{\FPdim \D} \,
R_{\D}\end{equation} where $\ff_!: R(\C) \to R(\D)$ is the algebra
map induced by $\ff$, and for all $Y \in \Lambda_\D$,
\begin{equation}\label{form-fpdimy}\sum_{X \in \Lambda_{\C}} m_Y(\ff(X))\, \FPdim X
= \frac{\FPdim \C}{\FPdim \D}\FPdim Y. \end{equation}
\end{lemma}

\begin{proof}Equation~\eqref{fdereg} is proved in \cite[Proposition 8.8]{ENO}, and~\eqref{form-fpdimy} is obtained by compairing the multiplicities of $Y$  in both sides of \eqref{fdereg}.
\end{proof}

Now for any simple object $X$ of $\C$, if $G$ is a right adjoint
of $\ff$, we have by adjunction $\Hom_\C(G(\1),X)\simeq
\Hom_\D(\1,\ff(X))$, hence $m_\1 F(X)=m_X G(\1)$. The same
equation holds if $G$ is a left adjoint. Thus
Equation~\eqref{form-fpdimy} for $Y=\1$ reads:
$\FPind(\ff)=\sum_{X \in \Lambda_\C} \FPdim X \,m_XG(\1)=\FPdim
G(\1)$.
\end{proof}

\begin{lemma}\label{fpdim-tf} Let $\ff: \C \to \D$ be a tensor
functor between fusion categories and $X$ be an object of $\C$. Then $X$ belongs to $\KER_\ff$ if and only if $\FPdim X =
m_\1\ff(X)$, and in this case $\ff(X) \simeq \1^{\FPdim X}$.
\end{lemma}

\begin{proof} An object $Y$ of $\D$ is trivial if and only if $\FPdim(Y)=m_\1(Y)$. Besides, $F$ preserves Frobenius-Perron dimensions. Thus
$X$ is in $\KER_\ff$ if and only if $\FPdim(X)= \FPdim F(X)=m_\1 F(X)$. \end{proof}

\subsection{Exact sequences of pointed categories}\label{pointed}
Recall that a \emph{pointed category} is a fusion category $\C$
whose simple objects are invertible, so that $G=\Lambda_\C$ is a
finite group for the tensor product called the \emph{Picard group}
of $\C$ and denoted by $\Pic\C$. If $\C$ is pointed then $\C$ is
equivalent to the category $\C(G, \alpha)$ of $G$-graded vector
spaces with associativity constraint given by a $3$ cocycle in a
class $\alpha \in H^3(F, \kk^\times)$. It is well-known that
this correspondance yields a bijection $(G,\alpha) \to
\C(G,\alpha)$ between pairs $(G,\alpha)$, where $G$ is a finite
group and $\alpha \in H^3(G, \kk^\times)$,  and pointed
categories over $\kk$ up to  equivalences of tensor categories.
Note that if a pointed category $\C$ is tensor equivalent to $\Rep{H}$ for some Hopf algebra $H$, then in fact  $H \simeq \kk^G$,
where $G=\Pic \C$, so that $\C \simeq \C(G,1)=\Rep{\kk^G}$.

\begin{proposition}\label{prop-pointed}
Let $\C' \overset{f}\to \C \overset{\ff}\to \C''$ be an exact
sequence of fusion categories over a field $\kk$, and assume that $\C$ is pointed.
Then $\C'$ and $\C''$ are pointed too, and we have an exact
sequence of groups
$$1 \to G' \to G \to G'' \to 1,$$
where $G, G', G''$ denote the Picard groups of $\C, \C',\C''$
respectively. Moreover, up to equivalence, such an exact
sequence of fusion categories is of the form
$$\C(G',1) \to \C(G,\infl(\alpha)) \to \C(G'',\alpha),$$
where $1 \to G' \to G \to G'' \to 1$ is an exact sequence of
finite groups, $\alpha$ is a cohomology class in
$H^3(G'',\kk^\times)$, and $\infl: H^3(G'',\kk^\times) \to
H^3(G,\kk^\times)$ denotes the inflation map.
\end{proposition}

\begin{proof}
The category $\C$ is of the form $\C(G,\beta)$, where $G=\Pic \C$
and $\beta \in H^3(G,\kk^\times)$. Being a full tensor subcategory
of $\C$, $\C'$ is of the form $\C(G',\beta')$ where $\beta'\in
H^3(G',\kk^\times)$ is the image of $\beta$ under the restriction
morphism. Since $\C'=\KER_\ff$, it is of the form $\Rep H$, $H$
being the Hopf algebra associated with our exact sequence of
fusion categories (see Section~\ref{fiberf}). Therefore $\beta'=1$
and $\C'=\C(G',1)$.

%

On the other hand, tensor functors preserve invertible objects.
Since the tensor functor $\ff: \C \to \C''$ is dominant and $\C$
is pointed, so is $\C''$. Therefore $\C''$ is of the form
$\C(G'',\alpha)$ for some finite group $G''$ and $\omega'' \in
H^3(G'', \kk^\times)$. Since $\ff$ preserves the monoidal
structures, we have $\beta=\infl(\alpha)$.

Conversely, given an exact sequence of finite groups $1 \to G' \to
G \to G'' \to 1$ and a class $\alpha \in H^3(G'',\kk^\times)$, we
have a tensor functor $\ff :\C=\C(G,\infl \alpha) \to
\C(G'',\alpha)$ such that $\KER_\ff=\C(G',1)$, hence an exact
sequence of fusion categories $\C(G,1) \to \C(G,\infl \alpha) \to
\C(G'',\alpha)$.
\end{proof}

\begin{remark} In an exact sequence of fusion categories
$\C' \to \C \to \C''$, such that $\C'$ and $\C''$ are pointed,
$\C$ need not be pointed. Indeed, let $H$ be a noncommutative
self-dual split semisimple Hopf algebra of dimension $p^3$, $p$ an odd
prime \cite{masuoka}. Then  $\Rep H$ is not  pointed. On the other hand, there is  a strictly exact
sequence  of Hopf algebras $\kk^{\mathbb Z_p} \to H \to
\kk(\mathbb Z_p \times \mathbb Z_p)$, which in turn induces an
exact sequence of fusion categories $\C(\mathbb Z_p \times \mathbb
Z_p, 1) \to \Rep H \to \C(\mathbb Z_p, 1)$ by Proposition~\ref{prop-schne}.
\end{remark}

\begin{remark}
Let $1 \to G' \to G \to G'' \to 1$ be an exact sequence of finite
groups  such that the inflation map $H^3(G'', \kk^\times) \to
H^3(G, \kk^\times)$ is not injective, and let $\alpha\in
H^3(G'',\kk^\times)$ be a non trivial element of its kernel. Then by
Proposition  \ref{prop-pointed}, we have an exact sequence of
fusion categories:
$$\C(G',1) \to \C(G, 1) \to \C(G'',\alpha)\,$$
and $\C(G'',\alpha)$ does not admit a fiber functor.
\end{remark}

Let $\C' \overset{f}\to \C \overset{\ff}\to \C''$ be an exact
sequence of fusion categories. If $\C''$ admits a fiber functor,
then so does $\C$. The converse is not true, as exemplified by the
previous remark. However, we have the following proposition.

\begin{proposition}
Let $\C' \overset{f}\to \C \overset{\ff}\to
\C''$ be an exact sequence of fusion categories. 
Then $\C$ is tensor equivalent to the category of representations
of a quasi-Hopf algebra if and only if $\C''$ is.
\end{proposition}

\begin{proof}
A fusion category is tensor equivalent to the representation
category of a quasi-Hopf algebra if and only if the
Frobenius-Perron dimensions of its objects are natural integers,
see \cite[Theorem 8.33]{ENO}. Hence the `if' Part because a tensor
functor preserves Frobenius-Perron dimensions (or by Tannaka
reconstruction), and the `if' Part  by \cite[Corollary 8.36]{ENO}.
\end{proof}

\subsection{Multiplicativity of Frobenius-Perron dimensions}

\begin{proposition}\label{fp-dim} Let $\C$, $\C'$, $\C''$ be fusion categories and $i : \C' \to \C$, $\ff : \C \to \C''$ be tensor functors. Assume that
$\ff$ is dominant, $i$ is full and $i(\C) \subset \KER_\ff$. Then
$\FPdim \C \ge \FPdim \C' \, \FPdim \C''$.  Moreover the diagram
$$\C' \overset{i}\to \C \overset{\ff}\to \C''$$ is an exact sequence if and only if
$\FPdim \C = \FPdim \C' \, \FPdim \C''$. If such is the case, then
for all simple object $Y \in \C''$,
\begin{equation}\label{fpdimy}\FPdim Y = \frac{1}{\FPdim \C'}\sum_{X \in
\Lambda_{\C}} m_Y (\ff(X)) \FPdim X. \end{equation}
\end{proposition}


\begin{proof} Firstly, notice that since the tensor functor $i_0 : \C' \to \KER_\ff$ is full, we have
$\FPdim \C' \le \FPdim \KER_\ff$, with equality if  and
only if $i_0$ is an equivalence. Thus, we may assume that
$\C'=\KER_\ff$, $i$ being the inclusion.

Equation~\eqref{form-fpdimy} of Lemma~\ref{FPdimdom}, applied to
the dominant functor $\ff : \C \to \C''$ and the simple object
$Y=\1$, yields:
\begin{equation}\label{y=1}\sum_{X \in \Lambda_{\C}} m_\1(\ff(X)) \,\FPdim X =
\frac{\FPdim \C}{\FPdim \C''}.
\end{equation}
Since $\C'=\KER_\ff$ is a full subcategory of $\C$, we may choose
representatives of classes of simple objects so that
$\Lambda_{\C'} \subset \Lambda_{\C}$. For $X \in \Lambda_{\C'}$,
we have $\ff(X)=\1^{\FPdim X}$, and therefore $\sum_{X \in
\Lambda_{\C'}} m_\1(\ff(X)) \,\FPdim X = \FPdim \C'$. Thus:
\begin{equation}\frac{\FPdim \C}{\FPdim \C''} = \FPdim \C' + E. \end{equation}
where $E=\sum_{X \in \Lambda_\C\backslash\Lambda_{\C'}}
m_\1(\ff(X))\, \FPdim X$. Now we have $E \ge 0$, and $E=0$ if and
only if $\ff$ is normal, hence the Proposition is proved.
\end{proof}

\begin{remark} Relation \eqref{fpdimy} implies the identity
$\ff_!(R_\C) = \frac{1}{\FPdim \C'} R_{\C''}$. \end{remark}


\begin{corollary}\label{equiv-ext} Consider a a diagram of tensor functors between fusion
categories with exact rows, and commutative up to tensor
isomorphisms:
$$\begin{CD}\C' @>{f}>> \C @>{\ff}>> \C'' \\
@VV{l}V @VV{\lambda}V @VV{r}V\\
\D' @>{g}>> \D @>{G}>> \D''.
\end{CD}$$
If $l$ and $r$ are equivalences, then $\lambda$ is an equivalence.
\end{corollary}

\begin{proof} Since $r$ and $l$ are equivalences, we may assume that they are identities, $\C'=\D'$, $\C''=\D''$, and the diagram commutes.
By Proposition \ref{fp-dim}, we have: $\FPdim \C = \FPdim \D$. Now
denote by $\E$ the dominant image of $\lambda$, that is, the full
subcategory of $\D$ whose objects are direct factors of objects
belonging to the image of $\lambda$. The tensor functor $\lambda$
factors as $j \lambda_0$, where $\lambda_0 : \C \to \E$ is
dominant, and $j : \E \to \D$ is a replete inclusion. In
particular, $\FPdim \E \le \FPdim \D$. In the diagram of tensor
functors
$$\xymatrix@C3em{\C' \ar[r]^{\lambda_0f}& \E \ar[r]^{Gj}& \C''}$$
the functor $Gj$ is dominant because $\ff=Gj\lambda_0$ is
dominant, $\lambda_0f$ is fully faithful because $g=j\lambda_0f$
and $j$ are fully faithful, and $\lambda_0f(\C'_1) \subset
\KER_{Gj}$ because $Gj\lambda_0f=\ff f$. By
Proposition~\ref{fp-dim}, we have $\FPdim \E \ge \FPdim \C' \,
\FPdim \C'' = \FPdim \D$. Hence $\FPdim \E=\FPdim \D$, so that
$\E=\D$ and therefore $\lambda$ is dominant. Applying again
Proposition~\ref{fp-dim}, this time to the sequence
$$\xymatrix{\KER_\lambda \ar[r]& \C \ar[r]^{\lambda}& \D}$$
we find that $\FPdim \KER_\lambda \le 1$, hence $\KER_\lambda
=\langle\1\rangle$, and the sequence is exact, so $\lambda$ is an
equivalence.
%
%
%
\end{proof}

\subsection{Functors of Frobenius-Perron index 2}\label{ind-2}


It is well known that if  $H$ is a split semisimple Hopf algebra
and $K$ is a Hopf subalgebra of index $2$, then $K$ is a normal
Hopf subalgebra and there is a cocentral exact sequence of Hopf
algebras $K \to H \to \kk\mathbb Z_2$. See \cite{kob-mas},
\cite[Corollary 1.4.3]{ssld}.  Hence an exact sequence of fusion
categories $\rep \mathbb Z_2 \to \Rep{H} \to \Rep{K}$.

We extend this result to general fusion categories. Recall that
the the Frobenius-Perron index  $\FPind \ff$ of a dominant tensor
functor  $\ff : \C \to \D$ between fusion categories is  the ratio
$\FPind \ff = \FPdim \C / \FPdim \C''$.

\begin{proposition}\label{indice2}  Let $\ff : \C \to \D$ be a dominant tensor functor of Frobenius-Perron index 2 between fusion categories.
Then  $\ff$ is normal, and we have an exact sequence of fusion
categories
$$\rep \mathbb Z_2 \to \C
\overset{\ff}\to \D.$$
\end{proposition}

\begin{proof}
Equation~\eqref{form-fpdimy} of Lemma~\ref{FPdimdom}, applied to
$Y=\1$, yields:
$$\sum_{X \in \Lambda_\C} m_\1(\ff(X)) \FPdim X = {\FPdim \C\over \FPdim \D}=2.$$
Since $\FPdim X \ge 1$ for any simple object $X$, we conclude that
there is exactly one element $J \in \Lambda_\C$, $J \not\simeq
\1$, such that $\ff(J)$ contains $\1$. Moreover we have $\FPdim
\ff(J)= \FPdim J=1$ so $\ff(J) \simeq \1$.  Thus $\ff$ is normal,
and we have an exact sequence of fusion categories $\KER_\ff \to
\C \to \D$. Note that $J$ is invertible. On the other hand,
$\KER_\ff$ is a pointed category whose group of invertibles is of
order $2$ (because $J$ is invertible). Therefore $\KER_\ff \simeq
\Rep H$, where $H$ is a split semisimple Hopf algebra $H$ of
dimension $2$, that is, $H\simeq \kk \mathbb Z_2$, hence the
Proposition holds.
\end{proof}

\section{Hopf monads and exact sequences}\label{sect-hm-es}

In this section, we study tensor functors and in particular, exact sequences of tensor categories in terms of Hopf monads. Recall that a tensor functor $\ff : \C \to \D$ is monadic if and only if it has a left adjoint $G$ (whose existence is equivalent to that of a right adjoint, see Section~\ref{rladj}). If such is the case,
the monad $T=\ff G$ of $G$ is a Hopf monad on $\D$, and $\C$ is tensor equivalent to the category $\D^T$ of $T$-modules in $\D$.

\begin{proposition}\label{dom-crit} Let $F : \C \to \D$ be a tensor functor between tensor categories over a field $\kk$, and assume $F$ admits a left adjoint $G$. Let $T=FG$ be the Hopf monad of $F$.
Then the following assertions are equivalent:
\begin{enumerate}[(i)]
\item The functor $\ff$ is dominant;
\item The unit $\eta$ of the monad $T$ is a monomorphism;
\item The monad $T$ is faithful;
\item The left adjoint of $\ff$ is faithful;
\item The right adjoint of $\ff$ is faithful.
\end{enumerate}
\end{proposition}

\begin{proof} The equivalence of the first four assertions results immediately from Lemma \ref{lem-mon-faith}, considering that, since $F$ is monadic, we may assume that $\C=\D^T$, $F$ is the forgetful functor $\U :\D^T \to \D$ and $G$ is
the free module functor. The right adjoint of $F$ is  $R\simeq \ldual{G(\rdual{?})}$, hence (iv) \ssi (v).
\end{proof}

\begin{proposition}\label{dom-crit2} Let $\ff :\C \to \D$ be a tensor functor between fusion categories over a field $\kk$, and let $T$ be the Hopf monad of $F$.
Then  $\FPind (\C:\D)\le \FPdim T(\1)$, and $\ff$ is dominant
if and only if
$$\FPind(\C:\D) = \FPdim T(\1).$$
If such is the case, then for all $X\in \Lambda_\C$, $\FPdim
T(X)=\FPdim T(\1)\, \FPdim X$.
\end{proposition}

\begin{proof} Let $G$ be a left adjoint of $\ff$.
If $\ff$ is dominant, we have $\FPdim \C= \FPdim T(\1) \FPdim \D$
by Proposition~\ref{inddom}, noting that $T(\1)=G(\1)$. In general
let $\E$ be the dominant image of $\ff$, that is, the full
subcategory of $\D$ whose objects are direct factors of objects
belonging to the image of $\ff$. Denote by $\ff_0$ the dominant
tensor functor $\C \to \E$ induced by $\ff$, and $i$ the full
embedding  $\E \subset \D$. Then $\Ll_{\mid \E}$ is left adjoint
to $\ff_0$ so we have $\FPdim \C= \FPdim T(\1) \FPdim \E$. We also
have $\FPdim \E \le \FPdim \D$, with equalitity if and only if
$\E=\D$, that is, $\ff$ is dominant. This proves both the
inequality and the equivalence of the Proposition.

Now assume $\ff$ is dominant. For $X\in \Lambda_\D$ we have by
Equation~\eqref{form-fpdimy} of Lemma~\ref{FPdimdom}:
\begin{align*}
\FPdim(X)&={\dim \D \over \dim \C} \sum_{Y \in \Lambda_\C} m_Y(\ff(X)) \FPdim Y\\
&= {1\over \FPdim T(\1)} \sum_{Y \in \Lambda_\C} m_X(G(Y)) \FPdim Y\quad\text{(by adjunction)}\\
&={1\over \FPdim T(\1)} \FPdim G(X)= {1\over \FPdim T(\1)} \FPdim
T(X),
\end{align*}
hence the last assertion of the Proposition holds.\end{proof}

\begin{definition}\label{normal} Let $\C$ be a tensor category. A \kt linear right exact Hopf monad $T$ on $\C$ is \emph{normal} if $T(\1)$ is a trivial
object of $\C$. \end{definition}

Recall that a tensor functor $\ff$ gives rise to a fiber functor
$\omega_\ff : \KER_\ff \to \vect_\kk, X \mapsto \Hom(\1,\ff(X))$,
hence a Hopf algebra $L(\omega_\ff)$ such that there exists a
canonical tensor equivalence $c:\KER_\ff \iso  \CoRep{
L(\omega_\ff)}$ (see Section~\ref{fiberf}).

\begin{lemma}
Let $T$ be  a \kt linear right exact normal Hopf monad
 on a
tensor category $\C$, with forgetful functor $\U : \C^T \to \C$.
Then the Hopf algebra $L=L(\omega_\U)$ is finite-dimensional, and we have a canonical tensor equivalence
$$\KER_\U \iso \CoRep{L}.$$
\end{lemma}

The Hopf algebra $L(\omega_\U)$ is called the \emph{induced Hopf algebra of the
normal Hopf monad $T$}.

\begin{proof} Firstly, we show that $L(\omega)$ is finite-dimensional. Since $T$ is normal, we have  $T(\langle\1\rangle) \subset \langle\1\rangle$, so that $T$ restricts to
a Hopf monad $T_{\mid\langle\1\rangle}$ on $\langle\1\rangle$, and
$\KER_\U={\langle\1\rangle}^{T_{\mid\langle\1\rangle}}$. Thus, we
may assume $\C=\langle\1\rangle$ (that is, $\C \simeq \vect_\kk$).
Then we have $\C_\U=\C^T$, and by Lemma~\ref{HM-triv}, there
exists a finite dimensional Hopf algebra $H$ on $\kk$  such that
$T=H \otimes ?$. We have a commutative triangle  of tensor
functors
$$\xymatrix{
\C^T \ar[dr]_{\omega_\U}\ar[rr]^{\sim\quad}&& \Rep{H}\ar[dl]\\
&\vect_\kk& }$$ whose horizontal is an equivalence, hence
$L(\omega_\U)\simeq \dual{H}$ is finite-dimensional.
\end{proof}

\begin{remark}\label{remtriv}
The induced Hopf algebra of a normal Hopf monad $T$ is the induced
algebra of the restriction of $T$ to the trivial tensor
subcategory $\langle\1\rangle \subset \C$.
\end{remark}

We have by definition:

\begin{proposition}
Let $F: \C\to \D$ be a tensor functor admitting a left adjoint.
Then $F$ is normal if and only if its Hopf monad $T$ is normal.
\end{proposition}
%

%
%
%
%

\begin{example}\label{ex-center}
Let $\C$ be a fusion category over a field $\kk$. It is shown in
\cite{bv-double} that the forgetful functor $\U: \Z(\C) \to \C$ is
monadic, with monad $Z$ defined by
$$Z(X)= \sum_{Y \in \Lambda_\C} \ldual{Y} \otimes X \otimes Y.$$ The unit of the monad
$Z$ is the $\1$-component inclusion $X\hookrightarrow Z(X)$; it is
therefore a monomorphism. Hence $\U$ is dominant by
Proposition~\ref{dom-crit}. However, $\U$ is normal if and only if
$\C$ is a pointed  category, because  $Z(\1)=\sum_{Y \in
\Lambda_\C} \ldual{Y} \otimes Y$ is trivial exactly in that case.
\end{example}


%
%

\subsection{Exact sequences and normal faithful Hopf monads}
The following theorem classifies extensions of tensor categories
in terms of normal faithful Hopf monads.

%

\begin{theorem}\label{t-mmod} Let
$\C'$, $\C''$ be tensor categories over a field $\kk$, and assume that $\C'$ is finite. Then the following
data are equivalent:
\begin{enumerate}
\item A normal faithful \kt linear right exact Hopf monad $T$ on $\C''$, with induced Hopf algebra $H$, endowed with a tensor equivalence $K :\C' \iso \CoRep{H}$;
\item An extension $\C' \to \C \to \C''$ of $\C''$ by $\C'$.
\end{enumerate}
\end{theorem}

\begin{proof}
Let $T$ be a normal faithful \kt linear right exact  Hopf monad on
$\C''$ with induced Hopf algebra $H$. Denote by $\U : {\C''}^T \to
\C''$ the forgetful functor. Let $q$ be a quasi-inverse of the
canonical tensor equivalence $\KER_\U \to \CoRep{H}$. According to
Corollary~\ref{cor-ex-from-T}, we have an exact sequence:
$$\CoRep{H}\overset{iq}\toto
(\C'')^{T} \overset{\U}\toto \C'',$$ hence an extension
$\C'\overset{iqK} \toto (\C'')^{T} \overset{\U}\toto \C''$ of
$\C''$ by $\C'$.

Conversely, let $\C' \overset{f}\toto \C \overset{\ff}\toto \C''$
be an extension of $\C''$ by $\C'$.  Then $\ff$ has a left adjoint
by Proposition~\ref{prop-crit-adj}.   Let $T$ be the Hopf monad of
$\ff$. Then $T$ is \kt linear, faithful (because $\ff$ is
dominant) and normal (because $\ff$ is normal). The tensor functor
$f$ induces a tensor equivalence $f_0: \C' \to \KER_\ff$, and we
also have a canonical tensor equivalence $c : \KER_\ff \iso
\CoRep{H}$, where $H$ is the induced Hopf algebra of the Hopf
monad $T$ of $\ff$, hence a tensor equivalence $K=c f_0 :\C' \iso
\CoRep{H}$.

These two constructions are mutually quasi-inverse. Indeed, if $T$
and $K$ are as in (1), then the Hopf monad of the forgetful
functor $\U: {\C''}^T \to \C''$ is $T$ and the reconstructed
tensor equivalence $\C' \to \CoRep{H}$ is $cqK \simeq K$.

On the other hand, given an extension $\C' \overset{f}\to \C
\overset{\ff}\to \C''$, and denoting by $T$ and $K$ the
corresponding Hopf monad and tensor equivalence, we have an
equivalence of extensions:
$$\xymatrix{
\C' \ar[d]_= \ar[r]^f & \C\ar[d]_{\kappa} \ar[r]^\ff & \C''\ar[d]^=\\
\C'\ar[r]_{iqK} & {\C''}^T \ar[r]_{\U} & \C'', }$$ where $\U$
denotes the forgetful functor, and $\kappa$ is the comparison
functor of the monadic functor $\ff$. Indeed $\kappa$ is a tensor
equivalence, we have $\U\kappa=\ff$, and $iqK=iqcf_0 \simeq
if_0=f$.
\end{proof}

Combining Theorem \ref{t-mmod} with Proposition \ref{indice2} we
obtain:

\begin{corollary}\label{index2} Let $\ff :\C \to \C''$ be a dominant tensor functor between fusion
categories $\C$ and $\C''$ such that $\FPind \ff = 2$. Then there
exists a \kt linear faithful normal semisimple Hopf monad $T$ on
$\C''$ having induced Hopf algebra  $\kk^{\mathbb Z_2}$, such that
$\C \simeq (\C'')^{T}$ as fusion categories. \qed \end{corollary}

\begin{example}\label{monad-secgroups}
Let $1 \to G'' \overset{\iota}\to G \overset{\pi}\to G' \to 1$ be
an exact sequence of finite groups.
Then we have an exact sequence of tensor categories
$$\xymatrix{\rep G' \ar[r]^{\pi^*}& \rep G \ar[r]^{\iota^*}& \rep G''}$$
as in Example \ref{sec-groups}. Let us describe the normal Hopf
monad $T$ on $\rep G''$ associated with this exact
sequence. We may assume without loss of generality that $G''$ is a
normal subgroup of $G$, and that $G'=G/G''$, $i$ being the
inclusion and $\pi$ the canonical surjection. The induction
functor  $\Ind^G_{G''}: \rep G'' \to \rep G$ is left
adjoint to the restriction functor $i^*=\Res^G_{G''}$. Let $Y$ be
a $\kk G''$-module. As a consequence of Mackey's Subgroup Theorem,
there is a natural isomorphism
$$\Res^G_{G''}\Ind^G_{G''} (Y) \simeq \oplus_{\gamma \in
G/G''} {}^\gamma Y,$$ where ${}^\gamma Y$ denotes the
$kG''$-module conjugated to $Y$ under the action of an element $g
\in G$ representing the class $\gamma$. See \cite[Remark
(10.11)]{curtis-reiner}. Then the Hopf monad $T$ is given, as an
endofunctor of  $\rep G'$,  by:
$$T(Y) = \oplus_{ \gamma \in G'} \,
{}^\gamma Y.$$ In fact, conjugation under elements $\gamma \in
G/G''$ defines an action of the group $G'$ on  $\rep G''$
by tensor autoequivalences. In subsequent subsections we will
study group actions on tensor categories in terms of Hopf monads.
\end{example}

\subsection{Semisimplicity}

A monad $T$ on a category  $\A$ is said to be
\emph{semisimple} if any $T$-module is a $T$-linear retract of a
free $T$-module, that is, of $(T(X), \mu_X)$, for some $X\in \A$.
We have the following analogue of Maschke's semisimplicity
criterion  for Hopf monads:

\begin{theorem}  \cite[Theorem 6.5.]{bv}\label{masch} Let $T$ be a Hopf monad on a rigid category $\C$. Then $T$ is semisimple if and only if there exists a
morphism $\Lambda : \1 \to T(\1)$ such that $\mu_\1 \Lambda=
\Lambda T_0$ and $T_0 \Lambda=\id_\1$.
\end{theorem}

\begin{corollary}\label{cormasch}
Let $T$ be a Hopf monad on a rigid category $\C$ and let $\C'
\subset \C$ be a full rigid subcategory of $\C$ such that $T(\C')
\subset \C'$. Then $T$ is semisimple if and only if its
restriction $T_{\mid \C'}$  to $\C'$ is semisimple.
\end{corollary}

\begin{proof}
Apply the theorem to $T$ and $T_{\mid \C'}$, which is a Hopf monad
on $\C'$.
\end{proof}

\begin{proposition}\label{propss}
Let $\C$ be a semisimple tensor category over a field $\kk$, and
let $T$ be a \kt linear Hopf monad on $\C$. Then the tensor
category $\C^T$ is semisimple if and only if the monad $T$ is
semisimple. In particular, if $\kk$ is algebraically closed and
$\C$ is a fusion category, then $\C^T$ is a fusion category if and
only if $T$ is semisimple.
\end{proposition}

\begin{proof}
Observe that, since $\C$ is semisimple, $T$ is exact. Let $\U :
\C^T \to \C$ be the forgetful functor, and $\Ll: \C \to \C^T$ be
the free module functor, defined by $\Ll(X)=(T(X),\mu_X)$, which
is left adjoint to $\U$. If $X$ is an object of $\C$, then
$\Ll(X)$ is a projective object of $\C^T$. Indeed, let $p :(Y,r)
\to \Ll(X)$ be an epimorphism in $\C^T$; then $p$ is an
epimorphism in $\C$, so it has a section $i : TX \to Y$ in $\C$
because $\C$ is semisimple. Then $r T(i)T(\eta_X)$ is a $T$-linear
section of $p$.

In particular, if $T$ is semisimple, any object of $\C^T$ is
projective, being a direct factor of a projective object, so
$\C^T$ is semisimple.

Conversely, assume $\C^T$ is semisimple, and let $(X,r)$ be a
$T$-module.  Then $r: T(X) \to X$ is an epimorphism because $r
\eta_X=\id_X$. It is also a morphism of $T$-modules from to
$(X,r)$, which is an epimorphism in $\C^T$ because the forgetful
functor $\U : \C^T \to \C$ is faithful exact. By semisimplicity of
$\C^T$, $r$ has a $T$-linear section so $(X,r)$ is a direct factor
of $\Ll(X)$. Hence $T$ is semisimple.
\end{proof}

\begin{corollary}\label{cor-ex-from-T}
Let $\C$ be a fusion category over an algebraically closed field
$\kk$, and let $T$ be a \kt linear faithful normal semisimple Hopf
monad on $\C$, with induced Hopf algebra $H$.  Then we have an
exact sequence of fusion categories
$$\CoRep{H} \to
(\C'')^{T} \overset{\U}\to \C''.$$
\end{corollary}

\begin{lemma} Let $\C$ be a tensor category over a field $\kk$, and let  $T$ be a \kt linear right exact normal dominant Hopf monad on $\C$ with induced Hopf algebra $H$.
Then $T$ is semisimple if and only if $H$ is cosemisimple.
\end{lemma}

\begin{proof}
Let $T_0$ be the restriction of $T$ to $\langle\1\rangle$. Then $T$ is semisimple if and only if $T_0$ is semisimple by
Corollary~\ref{cormasch}. On the other hand,
${\langle\1\rangle}^{T_0}=\KER_\U \simeq \CoRep{H}$, where $\U$ is the forgetful functor, so $T_0$ is semisimple
 $\iff$ ${\langle\1\rangle}^{T_0}$ is semisimple  (by Proposition~\ref{propss})
$\iff$ $H$ is cosemisimple.
\end{proof}

\begin{corollary} An extension of fusion categories over an algebraically closed field is a fusion category.
\end{corollary}

\begin{proof}
Let $\C' \overset{f}\toto \C \overset{F}\toto \C''$ be an exact sequence of fusion categories. Then $\C'$ is finite, so by Proposition~\ref{prop-crit-adj}, $F$ has a left adjoint and is
therefore monadic, and $\C$ is finite. Its monad $T$ is normal, dominant and its induced Hopf algebra $H$ is cosemisimple because $\C' \simeq \CoRep{H}$, so $T$ is semisimple by the previous lemma,
and $\C$ is semisimple. Since $\kk$ is algebraically closed, $\C$ is a fusion category.
\end{proof}

\subsection{Equivariantization}\label{equivariantization}
Let $\C$ be a tensor category over a field $\kk$. Denote by
$\underline \End_{\otimes}\C$ the monoidal category whose objects
are tensor endofunctors of $\C$, the morphisms being monoidal
natural transformations, the monoidal product being the
composition $\circ$ of tensor functors, and the unit object, the
identity functor $\id_\C$.

If $G$ is  a  group, denote by  $\underline G$ the strict monoidal
category whose objects are the elements of $G$ and morphisms are
identities, the monoidal product being the multiplication of $G$.

\begin{definition}An \emph{action of a group $G$ on a tensor category $\C$ (by tensor autoequivalences)} is a
strong monoidal functor
\begin{equation}\label{action}\rho: \underline G \to \underline \End_{\otimes}
\C.\end{equation} In other words, it consists in the following
data:
\begin{enumerate}
\item For each $g \in G$, a tensor endofunctor $\rho^g : \C \to \C$;
\item For each pair $g, h \in G$, a monoidal isomorphism $\rho^{g,h}_2 : \rho^g \rho^h \iso \rho^{gh}$;
\item A monoidal isomorphism  $\rho_0 : \id_\C \iso \rho^1$;
\end{enumerate}
such that for any $g, h, k$ in $G$ the following diagrams commute:
$$\xymatrix{
\rho^g \rho^h \rho^k \ar[r]^{\rho^{g,h}_2 \rho^k} \ar[d]_{\rho^g \rho^{h,k}_2}& \rho^{gh} \ar[d]^{\rho^{gh,k}_2} \rho^k & &\rho^g \ar[r]^{\rho^g \rho_0} \ar[d]_{\rho_0 \rho^g} \ar[rd]^= & \rho^g \rho^1 \ar[d]^{\rho^{g,1}_2}\\
\rho^g \rho^{hk} \ar[r]_{\rho^{g,hk}_2} & \rho^{ghk} & &\rho^1
\rho^g \ar[r]_{\rho^{1,g}_2} & \rho^g .}$$

Observe that if $G$ is a monoidal action of $G$ on $\C$, the
$\rho^g$'s are in fact tensor autoequivalences of $\C$,
$\rho^{g^{-1}}$ being  quasi-inverse to $\rho^g$ for all $g \in
G$.
%
%
\end{definition}


Let $\rho: \underline G \to \underline \Aut_{\otimes} \C$ be an
action of a group $G$ on a tensor category $\C$. A $G$-equivariant
object in $\C$ is a pair $(X, u)$, where $X$ is an object of $\C$,
and $u$ is a family $(u^g)_{g \in G}$, where for each $g \in G$,
$u^g:\rho^gX \to X$ is a morphism  satisfying:
\begin{equation}\label{deltau} u^g \rho^g(u^h) = u^{gh} {\rho^{g,
h}_{2_X}}\quad \text{for all $g,h \in G$}\quad \text{and}\quad
u_1{\rho_0}_X=\id_X.\end{equation} Note that the morphisms $u^g$
are then actually isomorphisms.

A $G$-equivariant morphism $f: (X, u) \to (Y, v)$ between
$G$-equivariant objects is a morphism $f: X \to Y$ in $\C$ such
that $f u_g = v_g f$ for all $g \in G$.

The $G$-\emph{equivariantization} of $\C$, denoted $\mathcal C^G$,
is, by definition, the category of $G$-equivariant objects and
$G$-equivariant morphisms \cite{agaitsgory, fw, nik, tambara}. It
is a tensor category, with monoidal product defined as follows: if
$(X, u)$ and $(Y, v)$ are $G$-equivariant objects, then
$$(X,u) \otimes (Y,v)=(X\otimes Y, w),\quad\text{where}\quad w=\bigl(w^g=(u^g \otimes v^g){{\rho^g_2}^{\,-1}_{X,Y}}\bigr)_{g\in G},$$
the unit object being $(\1,({\rho^g_0}^{\,-1})_{g \in G})$.

Moreover if $\C$ is a fusion category, $G$ is finite, $\kk$ is
algebraically closed and $\car(\kk)$ does not divide the order of
$G$, then $\C^G$ is a fusion category. In that case, it is shown
in \cite{nik} that $\C^G$ is dual to a crossed product fusion
category $\C \rtimes G$ with respect to the indecomposable module
category $\C$.

\begin{example}
Let $G$ be a group and let $\rho$ be the trivial action of $G$ on
$\vect_\kk$. Then $(\vect_\kk)^G=\rep G$.
\end{example}

\begin{definition}
A tensor functor $\ff : \C \to \D$ between tensor categories is an \emph{equivariantization}
if there exists a finite group $G$ acting on $\D$ by tensor equivalences, and a tensor equivalence
$\C \simeq \D^G$ over $\D$.
\end{definition}

\subsection{Characterization of equivariantizations in terms of  Hopf monads}\label{monad-groupaction}
In this subsection we show that group actions on tensor categories
and equivariantization can be interpreted in the language of Hopf
monads.

\begin{definition}\label{def-cocom} A normal \kt linear right exact Hopf monad $T$ on a tensor category $\C$ is \emph{cocommutative} if
for any morphism $g : T(\1) \to \1$ and any object $X$ of $\C$,
the following square is commutative:
$$\xymatrix@C5em{
T(X) \ar[r]^{T_2(X,\1)\quad} \ar[d]_{T_2(\1,X)} & T(X)\otimes T(\1)\ar[d]^{\id_{T(X)} \otimes \,g}\\
T(\1) \otimes T(X)\ar[r]_{g \,\otimes\,\id_{T(X)}} & T(X). }$$
\end{definition}

\begin{theorem}\label{monact} Let $\C$ be a  tensor
category over a field $\kk$, and let $\rho$ be an action of a
finite group $G$ on $\C$ by tensor autoequivalences. Then:
\begin{enumerate}
\item The \kt linear exact endofunctor
$$\TT^\rho= \bigoplus_{g \in G} \rho^g$$
admits a canonical structure of Hopf monad on $\C$;
\item There is a canonical isomorphism of categories:
$$\C^G \simeq \C^{\TT^\rho}$$ over $\C$, where $\C^G$ denotes the equivariantization of $\C$ under $G$;
\item
The Hopf monad $\TT^\rho$ is faithful, normal, and cocommutative.
\item The induced Hopf algebra of $T$ is $\kk^G$. In particular $\TT^\rho$ is semisimple if and only if  $\car(\kk)$ does not divide the order of $G$.
\end{enumerate}
\end{theorem}

\begin{proof} The endofunctor $\TT^\rho$ is \kt linear exact by construction.
Define natural transformations
$$\mu : (\TT^\rho)^2 =\bigoplus_{g,g'} \rho^g\rho^{g'} \to \TT^\rho=\bigoplus_{h \in G} \rho^h \quad \text{and}\quad \eta: \id_\C \to \TT^\rho=\bigoplus_{g\in G}\rho^g$$ componentwise by the collection of morphisms $\rho^{g,g'}_2 : \rho^g\rho^{g'} \to \rho^{gg'}$ and by the morphism $\rho_0: \id_\C \to \rho^1$ respectively. The axioms of a group action imply that
$(\TT^\rho,\mu,\eta)$ is a monad on $\C$. Given an object $X$ of
$\C$, the canonical bijection $$\prod_{g \in G} \Hom_\C(\rho^gX,X)
\iso \Hom_\C(\bigoplus_{g \in G} \rho^gX,X)$$ restricts to a
bijection between families $u=(u^g)$ such that $(X,u)$ is an
object of $\C^G$ on one hand, and actions $r: \TT^\rho(X) \to X$
of the monad $\TT^\rho$ on $X$, and this induces an isomorphism of
categories $\kappa : \C^G \to \C^{\TT^\rho}$ over $\C$. In
particular, $\C^{\TT^\rho}$ is a tensor category over $\kk$, and
the forgetful functor $\U : \C^{\TT^\rho} \to \C$ is a tensor
functor. This implies that $\TT^\rho$ is a Hopf monad on $\C$. The
comonoidal structure of $\TT^\rho$ is as follows:
$$\TT^\rho_2(X,Y): \bigoplus_{g \in G} \rho^g(X \otimes Y) \to \bigoplus_{g',g'' \in G} \rho^{g'}X \otimes \rho^{g''}Y\quad\text{and}\quad \TT^\rho_0:  \bigoplus_{g \in G} \rho^g\1\to \1$$
are given componentwise by the strong (co)monoidal structure of
the tensor functors $\rho^g$, that is, by
$${\rho^g_2}^{\,-1}_{X,Y}: \rho^g(X \otimes Y) \iso \rho^gX \otimes \rho^gY\quad\text{and}\quad{\rho^g_0}^{\,-1} : \rho^g\1 \iso \1.$$
Hence Parts (1) and  (2). Now $\TT^\rho$ is faithful because
$\eta$ is a monomorphism.  It is normal because we have
$\TT^\rho(\1) \iso \1^G$. The cocommutativity of $\TT^\rho$
results from the fact that the endofunctor $\rho^g$ being strong
monoidal for  all $g \in G$,  for any object $X$ of $\C$ the
diagram of isomorphisms:
$$\xymatrix@C6em{
\rho^gX \ar[r]^{{\rho^g_2}^{\,-1}_{X,\1}} \ar[d]_{{\rho^g_2}^{\,-1}_{\1,X}}  & \rho^gX \otimes \rho^g\1 \ar[d]^{\id_{\rho^gX}\otimes {\rho^g_0}^{\,-1}}\\
\rho^g\1 \otimes \rho^gX \ar[r]_{{\rho^g_0}^{\,-1}\otimes
\id_{\rho^gX}} &  \rho^gX }$$ is commutative, hence Part (3).

Let $L$ be the induced Hopf algebra of $\TT^\rho$. As noted
previously (see Remark~\ref{remtriv}), $L$ is also the induced
Hopf algebra of the restriction of $\TT^\rho$ to
$\langle\1\rangle$, that is, of $\TT^{\rho'}$, where $\rho'$ is
the restriction to $\langle\1\rangle$ of the action of $G$ on
$\C$. In order to compute $L$ we may therefore assume that $\C$ is
a trivial fusion category (that is, all objects of $\C$ are
trivial).  In that case, the category of  tensor endofunctors of
$\C$ is equivalent to the point; we may consequently  assume that
$\rho$ is the trivial action. We have a commutative square of
tensor functors:
$$\xymatrix@C=3em{
\C^G \ar[rr]^{\sim\qquad} \ar[dr]_{\Hom_\C(\1,\U)}&&
{\vect_\kk}^G=\rep G\ar[dl] \\  & \,\vect_\kk&, }$$ whose
horizontal arrow is an equivalence, hence
$L=L(\Hom_\C(1,\U))\simeq \kk^G$. In particular,
$\TT^\rho$ is semisimple $\iff$ $\rep G$ is semisimple $\iff$
$\car(\kk)$ does not divide the order of $G$. Hence Part (4).
\end{proof}

The Hopf monad $\TT^\rho$ is called the \emph{monad of the group
action $\rho$.} We also denote it by $\TT^G$ when the action is
clear from the context.
%
%
%

%


\begin{corollary}\label{exact-gequiv} An action of a finite group $G$  on a tensor category $\C$ by tensor autoequivalences gives rise to an exact sequence of tensor categories:
$$\rep G \to \C^G\to
\C.$$ 
It is is an exact sequence of fusion categories if $\kk$ is algebraically closed, $\C$ is a fusion category and $\car(\kk)$ does not divide the order of $G$.
\end{corollary}

\begin{proof} Results from Theorem~\ref{monact} and Corollary~\ref{cor-ex-from-T}.\end{proof}

\begin{remark}
The fusion category $\C_p$ constructed in \cite[4.1]{nik} is
group-theoretical and admits an action of the group $\mathbb Z_2$,
such that $\C_p^{\mathbb Z_2}$ is not group-theoretical
\cite[Corollary 4.6]{nik}. The resulting  exact sequence $\rep
\mathbb Z_2 \to \C_p^{\mathbb Z_2} \to \C_p$ shows that an
extension of group-theoretical categories need not be
group-theoretical.
\end{remark}

The converse of Theorem~\ref{monact} is true:

\begin{theorem}\label{converse-equiv}
Let $T$ be a \kt linear right exact faithful normal cocommutative
Hopf monad on a tensor category $\C$ over a field $\kk$, whose
induced Hopf algebra $H$ is split semisimple. Then $H$ is
isomorphic to $\kk^G$ for some finite group $G$,  and  there
exists an action of  $G$ on $\C$ by tensor autoequivalences $\rho:
\underline G \to \underline \End_{\otimes} \C$ such that $T \simeq
\TT^\rho$.
\end{theorem}

\begin{proof} The restriction $T_0$ of $T$  to $\langle\1\rangle$ is isomorphic to $L \otimes ?$ as a Hopf monad,
where $L$ is the Hopf algebra $H^*$.  Since $T$ is cocommutative,
so is $L$, that is, $H$ is commutative. Being split semisimple,
$H$  is of the form $\kk G$ for some finite group $G$.

For $g \in G$, denote by $e^g : T(\1) \to \1$ the morphism
corresponding to the map $\kk G \to \kk$, $h \mapsto \delta_{g,h}$
via an isomorphism $T_0 \simeq \kk G \otimes ?$

The morphisms $e^g : T (\1) \to \1$ satisfy the following
equations:
\begin{align*}
&(1)\quad (e^g \otimes e^h)T_2(\1,\1)=\delta_{g,h} e^g,
&(2)\quad \sum_{g} e^g= T_0,\\
&(3)\quad e^g \mu_\1=\sum_{g'g''=g} e^{g'} T(e^{g''}) , &(4)\quad
e^g \eta_\1 =\delta_{g,1} \id_\1,
\end{align*}
which reflect the Hopf algebra structure of $\kk G$.

For $g \in G$, define a natural endomorphism $\pi^g$ of $T$ by
setting $$\pi^g_X=(e^g \otimes \id_{T(X)})T_2(\1,X), \quad
\text{for $X$ object of $\C$.}$$ Note that we also have
$\pi^g_X=(\id_{T(X)} \otimes e^g)T_2(X,\1)$, $T$ being
cocommutative.

Using Equations (1) and (2) above and the comonoidality of $T$,
one verifies easily the following equations: \begin{align*}
&(5)\quad \pi^g_X \pi^h_X= \delta_{g,h} \pi^g_X, \quad (6)\quad \sum_g \pi^g_X=\id_{T(X)},\\
&(7)\quad T_2(X,Y)\pi^g_{X\otimes Y}= (\pi^g_X \otimes \id_{T(Y)}) T_2(X,Y)=(\id_{T(X)}\otimes \pi^g_Y) T_2(X,Y),\\
\end{align*}
the last equation resulting from cocommutativity of $T$.

By (5) and (6), for $X$ object of $\C$ the family $(\pi^g_X)_{g
\in G}$ is a complete orthogonal system of idempotents of $T(X)$ .
Denote by $\rho^g_X$ the image of the idempotent $\pi^g_X$. This
defines an endofunctor $\rho^g$ of $\C$, and we have
$$T= \bigoplus_{g \in G} \rho^g.$$
The point is now to show that $\rho : g \mapsto \rho^g$ is an
action of $G$ on $\C$ by tensor autoequivalences, so that
$T=\TT^\rho$. From Equations (3) and~(4), we deduce
$$(8) \quad \pi^g_X=\sum_{g'g''=g} \mu_X \pi^{g'}_{T(X)} T(\pi^{g''}_X), \quad (9) \quad \pi^g_X \eta_X=\delta_{g,1} \eta_X.$$
Indeed, we have
\begin{align*}
&\pi^g_X \mu_X = (e^g \otimes \id_{T(X)})T_2(\1,X)\mu_X
=(e^g \mu_\1\otimes \mu_X)T_2(T(\1), T(X))T(T_2(\1,X))\\
&=\sum_{g'g''=g}(e^{g'}T(e^{g''})\otimes \mu_X)T_2(T(\1),
T(X))T(T_2(\1,X))
=\sum_{g'g''=g} \mu_X \pi^{g'}_{T(X)} T(\pi^{g''}_X);\\
&\pi^g_X \eta_X = (e^g \otimes \id_{T(X)})T_2(\1,X)\eta_X=e^g \eta_\1\otimes \eta_X=\delta_{g,1} \eta_X.\\
\end{align*}
Thus $\mu : \bigoplus_{g,h} \rho^g\circ \rho^h \to \bigoplus_k
\rho^k$ is given componentwise by morphisms
$$\rho^{g,h}_2 : \rho^g \circ \rho^h \to \rho^{gh},$$
and $\eta : \id_\C \to \bigoplus_g \rho^g$, by a morphism $\rho_0:
\id_\C \to \rho^1$. Since $T$ is a monad, we have:
$$\rho^{g,hk}_2\rho^g\rho^{h,k}_2=\rho^{gh,k}_2\rho^{g,h}_2\rho^k\quad\text{and}\quad \rho^{g,1}\rho^g \rho_0=\id_{\rho^g}=\rho^{1,g}\rho_0 \rho^g.$$
Using Equations~(5) and (7), we also have
$$(\pi^g_X \otimes \pi^h_Y)T_2(X,Y)=\delta_{g,h}T_2(X,Y)\pi^g_{X\otimes Y}.$$
Thus $T_2(X,Y) : \bigoplus_g \rho^g(X\otimes Y) \to
\bigoplus_{h,k} \rho^h(X) \otimes \rho^k(Y)$ is given
componentwise by morphisms
$$f^g_2(X,Y) : \rho^g(X \otimes Y) \to \rho^g(X) \otimes \rho^g(Y).$$
Lastly, $T_0 : \bigoplus_g \rho^g(\1) \to \1$ is given by
morphisms $f^g_0 : \rho^g(\1) \to \1$. The fact that $T$ is a
bimonad implies that $(\rho^g,f^g_2,f^g_0)$ is a comonoidal
endofunctor of $\C$ for all $g \in G$, and that the natural
transformations $\rho^{g,h}_2: \rho^g \rho^h \to \rho^{gh}$ and
$\rho_0 : \id_\C \to \rho^1$ are comonoidal.

Next we show that the structure morphisms $\rho^{g,h}_2$,
$\rho_0$, $f^g_2$, $f^g_0$ are isomorphisms. The left fusion
operator $H^l$ of the bimonad $T$, introduced in~\cite{blv} and
defined as
$$H^l_{X,Y}=(\id_{T(X)} \otimes \mu_Y)T_2(X,T(Y)) : T(X\otimes T(Y)) \to T(X) \otimes T(Y),$$
is an isomorphism by \cite[Theorem 3.10]{blv}, because $T$ is a
Hopf monad. We have  $T(X\otimes T(Y))=\bigoplus_{g,h \in G}
\rho^g(X \otimes \rho^hY)$
and $T(X) \otimes T(Y)=\bigoplus_{m,n\in G}  \rho^mX \otimes
\rho^nY$,
%
%
and $H^l_{X,Y}$ is defined componentwise by isomorphisms
$$\omega^{g,h}_{X,Y}=(\id_{\rho^gX} \otimes {\rho^{gh}_2}_Y)f^g_2(X,\rho^h(Y)): \rho^g(X \otimes \rho^h(Y))\iso \rho^gX \otimes \rho^{gh}Y$$
Now for $g \in G$, $f^g_0$ is an isomorphism because $T_0 =\sum_g
e_g$. So
$${\rho^{gh}_2}_Y=(f^g_0 \otimes \id_{\rho^{gh}_Y}) \,\omega^{g,h}_{\1,Y}$$
is an isomorphism too. Now we check that $\rho_0 : \id_\C \to
\rho^1$ is an isomorphism. Observe first that $\rho_0$ is a
monomorphism because, $T$ being faithful, $\eta$ is a
monomorphism, and $\eta$ factors through $\rho_0$. In particular,
$\rho^1$ is faithful; it is also \kt linear right exact since it
is a direct summand of $T$. We have ${\rho^{11}_2}
\rho^1(\rho_0)=\id_{\rho^1}$, so
$\rho^1(\rho_0)={{\rho^{11}_2}^{-1}}$ is an isomorphism. Since
$\rho^1$ is faithful right exact,  $\rho_0$ is an epimorphism. The
category $\C$ being abelian, $\rho_0$ is an isomorphism.

In particular $\rho^g$ is a \kt linear autoequivalence of $\C$,
with quasi-inverse $\rho^{g^{-1}}$.

Now $f^g_2(X,Y)$ is an isomorphism, because  $Y \simeq \rho^1Y$,
$\omega^{g,1}$ and $f^g_2$ are isomorphisms. Thus $(\rho^g,
f^g_2,f_0)$ is a strong comonoidal  functor, that is, $(\rho^g,
(f^g_2)^{-1}, (f^g_0)^{-1})$ is a strong monoidal functor. It is
therefore a tensor autoequivalence of $\C$. We have shown that
$\rho$ is an action of $G$ on $\C$ by tensor autoequivalences, and
$T=\TT^\rho$.
\end{proof}

\begin{corollary}\label{corequiv} Let $\ff : \C \to \D$ be a tensor functor between tensor categories over a field $\kk$ admitting a left adjoint.
Then $\ff$ is an equivariantization if and only if $\ff$ is dominant normal, its Hopf monad $T$ is cocommutative and its induced Hopf algebra is split semisimple.
\end{corollary}

\begin{proof}
The functor $\ff$ is monadic, with monad $T$. We may therefore assume
$\C=\D^T$, $\ff$ being the forgetful functor $\U$. Let $L$ be the
induced Hopf algebra of $T$.
divisible by $\car \kk$. We conclude by
Theorem~\ref{converse-equiv}, noting that if $L$ is commutativ, then it is the function algebra of a finite group if and only if it is  split semisimple.
\end{proof}

\begin{example} Let $\Gamma$ be a finite group and let $L$ be a finite-dimensional Hopf algebra over an algebraically closed field $\kk$ endowed with a $\Gamma$-graduation $L=\bigoplus_{g \in \Gamma} L_g$ such that $L_g \neq 0$ for all $g \in \Gamma$. Let $H=\dual{L}$. The $\Gamma$-graduation on $L$ translates into an injective  Hopf algebra morphism
$i : \kk^\Gamma \to H$ whose image is central in $H$. This
morphism is characterized by:
$$\langle i(\varphi), \lambda \rangle = \sum_{g \in \Gamma  } \varphi(g) \varepsilon (\lambda_g)$$
for all $\phi \in \kk^{\Gamma }$, $\lambda \in L$, $\varepsilon$
denoting the counit of $L$. We have $H=\bigoplus_{g\in \Gamma}
H_g$, where $H_g=\dual{L}_g$.
The dominant tensor functor  $\Res_{\kk^\Gamma}^H :   \Rep{H}
\to \Rep{\kk^\Gamma}$ is monadic,  with Hopf monad $\T=
\Res^H_{\kk^\Gamma}\Ind_{\kk^\Gamma}^H$. The tensor category
$\Rep{\kk^\Gamma}$ is the pointed category $\C(\Gamma)$ of
$\Gamma$-graded vector spaces, whose simple objects are indexed by
the elements of $\Gamma$, and we have $\T(g)=H_g \otimes g$; in
particular $\T$ is normal, with induced Hopf algebra $H_1$. In
particular, if $L$ is cosemisimple, $H$ is semisimple and we have
an exact sequence of fusion categories
$$\Rep{H_1} \to\Rep{H} \to \C(\Gamma).$$
Moreover, $T$ is cocommutative if and only if $L_1$ is contained
in the center of $L$. Now let $\Gamma$ act on a finite group $G$
in a non-trivial way by $\Gamma \times G \to G$, $(x, g) \to x.g$,
and let $L$ be the abelian extension $L=\kk^G \# \kk \Gamma$
corresponding with this action. Then the multiplication and
comultiplication in $L$ are given by:
$$(e_g\# x)(e_h\# y) = \delta_{g, x . h} e_g\# xy, \quad \Delta(e_g\# x) = \sum_{st = g}e_s \# x \otimes e_t \# x = \Delta(e_g)
\Delta(x),$$ where $e_g \in \kk^G$ are defined as $e_g(h) =
\delta_{g,h}$, $g, h \in G$. Thus $L_1 = \kk^G$ is a Hopf subalgebra
which is not central in $L$ and $\kk^{\Gamma}$ is a central Hopf
subalgebra in $H = \dual{L}$.

Then $L$ is a cosemisimple $\Gamma$-graded Hopf algebra, and
$L_1=\kk^G$ is commutative, but not central in $L$, hence for
$H=\dual{L}$ there is  an
exact sequence of fusion categories
$$\rep G \to \Rep{H}  \to \C(\Gamma)$$
which is not induced by an action of the group $G$ on $\C(\Gamma)$ by tensor autoequivalences.
\end{example}

\subsection{The braided case - Modularization revisited}

\begin{definition} A bimonad  $T$ on a braided category $\C$ is \emph{braided} if the following diagram is commutative:
$$\xymatrix@C5em{
T(X \otimes Y) \ar[r]^{T_2(X,Y)} \ar[d]_{T(c_{X,Y})}& T(X) \otimes T(Y)\ar[d]^{c_{T(X),T(Y)}}\\
T(Y \otimes X) \ar[r]^{T_2(Y,X)}& T(Y) \otimes T(X) }$$ for any
$X, Y$ objects of $T$, where $c$ denotes the braiding of $\C$.
\end{definition}
\begin{remark} This is equivalent to saying that $R_{X,Y}=(\eta_Y \otimes \eta_X)c_{X,Y}$ defines a $R$-matrix for $T$. \end{remark}

\begin{proposition} Let $T$ be a bimonad on a braided category $\C$, with forgetful functor  $\U : \C^T \to \C$.
There exists a braiding on $\C^T$ such that $\U$ is braided if and
only if $T$ is braided.
\end{proposition}

\begin{proof} Since $\U$ is faithful, a braiding $\tilde{c}$ on $\C^T$ such that $\U$ is braided is necessarily given by
$$\tilde{c}_{(M,r),(N,s)} = c_{M,N}\quad \mbox{for any $T$-modules $(M,r)$, $(N,s)$.}$$
Now the morphism $\tilde{c}$ so defined is $T$-linear if and only
if $T$ is braided; and if such is the case, it is a braiding on
$\C^T$.
\end{proof}

\begin{proposition}\label{braidcocom} A braided \kt linear right exact normal Hopf monad on a braided tensor category $\C$ over a field $\kk$ is cocommutative.
\end{proposition}
\begin{proof} Denoting by $c$ the braiding of $\C$, we have $T_2(X,\1)=c_{T(1),T(X)} T_2(\1,X)$ for any object $X$ of $\C$, hence
$(g \otimes\, \id_{T(X)})T_2(\1,X)=(\id_{T(X)} \otimes\,
g)T_2(X,\1)$ for any $g: T(\1) \to \1$ by functoriality of $c$.
\end{proof}

\begin{corollary}\label{brequiv}
Let $\ff : \C \to \D$ be a braided tensor functor between braided
tensor categories over an algebraically closed field $\kk$
admitting a left adjoint, and assume that $\ff$ is normal and
dominant.  Assume that $\car \kk$ does not divide $\FPind \ff$.
Then there exists a finite group $G$ acting on $\D$ by braided
tensor autoequivalences and a braided tensor equivalence $\C \to
\D^G$ over $\C$.
\end{corollary}

\begin{proof}
The tensor functor $\ff$ being monadic, with Hopf monad $T$, we
have an equivalence of tensor categories $\C \to \D^T$ over $\C$.
The Hopf monad $T$ is \kt linear right exact, faithful, normal,
and by Proposition~\ref{braidcocom}, it is cocommutative. By
Corollary~\ref{corequiv}, $T$ is the Hopf monad of an action of a
finite group $G$ on $\D$ by tensor autoequivalences and we have
$\C \simeq \D^G$. Moreover $T$ is braided, which means that $G$
acts by braided autoequivalences.
\end{proof}

\begin{example}\label{equivalencia} Let $1 \to G'' \overset{\iota}\to G \overset{\pi}\to G' \to
1$ be an exact sequence of finite groups, with its associated exact sequence of tensor categories
 $$\rep G' \overset{\pi^*}\toto \rep G
\overset{\iota^*}\toto \rep G''.$$ The tensor functor $i_* : \rep G \to \rep G''$ is symmetric, so its monad $T$ is cocommutative
and $i^*$ is an equivariantization functor. In fact, $T$ is the monad on $\rep G''$ introduced in Example~\ref{monad-secgroups}, and it is
the monad of the action of $G'$ on $\rep G''$ by conjugation. The tensor equivalence  $\rep G \simeq (\rep G'')^{G'}$ is a special case of the one established in
\cite{ext-ty} for cocentral extensions of finite dimensional Hopf
algebras. \end{example}

\begin{example}\label{exmodable}
Let $\C$ be a premodular category, and let $\T \subseteq \C$ be
the category of transparent objects of $\C$. Assume $\C$ is
modularizable,  and let $\ff : \C \to \widetilde \C$ be its
modularization (see \cite{bruguieres}). The modularization functor
is dominant and normal, and we have $\KER_\ff= \T $ (see
\cite[Propositions 2.3 and 3.2]{bruguieres}) hence an exact
sequence of fusion categories:
$$\xymatrix{\T \ar[r]& \C \ar[r]^\ff&
\widetilde \C.}$$ Moreover, $\ff$ is a braided functor; it is
therefore an equivariantization by Corollary~\ref{brequiv}. In
fact, $\T$ is a tannakian category, so that we have a symmetric
tensor equivalence $\T \simeq \rep G$,  $G$ being a finite group,
$G$ acts on $\tilde{\C}$ and $\C=\tilde{\C}^G$. Modularization is
therefore a special case of the de-equivariantization procedure
described in \cite[2.6]{eno2}.
%
%
%
\end{example}

\begin{proposition}
Let $\D$ be a modular category over an algebraically closed field
$\kk$ of characteristic $0$, with twist $\theta$. The following
data are equivalent:
\begin{enumerate}[(A)]
\item A premodular category $\C$ and a modularization functor $\C \to \D$;
\item A \kt linear faithful normal  braided semisimple Hopf monad $T$ on $\D$ preserving the twist, that is, such that $\theta_T=T(\theta)$.
\end{enumerate}
\end{proposition}

\begin{proof} Let $T$ be a \kt linear faithful normal braided semisimple Hopf monad on $\D$, and let $\U : \D^T \to \D$ be the forgetful functor.
Then $\D^T$ is a braided fusion category, and $\U$ is a dominant
braided tensor functor. One verifies easily that the condition
that $T$ preserves the twist is equivalent to saying that there
exists a twist $\tilde{\theta}$ on $\D^T$ which is preserved by
$\U$. So if $T$ preserves the twist, $\D^T$ is premodular, and
$\U$ is a modularization. Conversely, let $\ff : \C \to \D$ be a
modularization. Its monad $T$  is a \kt linear faithful
normal braided semisimple Hopf monad, and since $\ff$ preserves
twists, $T$ preserves the twist of $\D$.
\end{proof}

\section{Exact sequences and  commutative central algebras}\label{sect-es-cca}

The modularization $F : \C \to \tilde{\C}$ of a modularizable
premodular category $\C$  is constructed in \cite{bruguieres} as
the free module functor $\C \to \mod_\C(A)$, $A$ being a
commutative algebra in the braided category $\C$. More precisely,
$A$ is a trivializing algebra of the full subcategory $\T \subset
\C$ of transparent objects of $\C$.

In this section we show that, more generally, any dominant functor
$\ff : \C \to \D$ between tensor categories admitting an exact right adjoint is, up to tensor
equivalence, a free module functor $\C \to \mod_\C A$, $A$ being a
certain commutative algebra in the center of $\C$ called the
induced central algebra of $\ff$.  Such a functor is normal if and
only if $A$ is self-trivializing.

\subsection{Induced central algebra of a tensor functor}

If $A$ is an algebra in a tensor category $\C$ over $\kk$, with
product $m: A \otimes A \to A$ and unit $u : \1 \to A$,  we denote
by $\MOD{\C}{A}$  the abelian \kt linear category of right
$A$-modules in $\C$. The forgetful functor $V_A: \MOD{\C}{A}\to
\C$ is \kt linear exact, and has a left adjoint, namely the
\emph{free $A$-module functor} $F_A : \C \to \MOD{\C}{A}$ defined
by $X \mapsto (X \otimes A, \id_X \otimes m)$.

We say that $A$ is \emph{semisimple} in $\C$ if every right
$A$-module in $\C$ is a direct factor of $F(X)$ for some object
$X$ of $\C$. Note that if  $A$ is semisimple and $\C$ is
semisimple, then $\MOD{\C}{A}$ is semisimple too.

A \emph{central algebra of $\C$} is an algebra $A$ in $\C$ endowed
with a half-braiding $\sigma: A \otimes \id_\C \iso \id_C \otimes
A$ such that the pair $(A,\sigma)$ is an algebra in the
categorical center $\Z(\C)$ of $\C$. This means that the product
$m$ and unit $u$ of $A$ are morphisms of half-braidings, that is:
$$\sigma_X(m \otimes \id_X)=(\id_X \otimes m)(\sigma_X \otimes \id_A)(\id_A \otimes \sigma_X) \quad\mbox{and}\quad \sigma_X(u \otimes \id_X)=\id_X \otimes u_.$$
A central algebra $(A,\sigma)$ is \emph{commutative} if
$m\sigma_A=m$.

Now let $(A,\sigma)$ be a commutative central algebra of $\C$. We
define a tensor product $\otimes_A$ on $\MOD{\C}{A}$, as follows.
Given two right $A$-modules $M$ and $N$, with $A$-actions $r : M
\otimes A \to M$ and $s: N \otimes A \to N$, the tensor product
$M\otimes_A N$ is the coequalizer of the pair of morphisms
$$(r \otimes \id_N,\id_M \otimes s \sigma_N) : M \otimes A \otimes N \rightrightarrows M \otimes N.$$
It is a right $A$-module, with action $t : M \otimes_A N \otimes A
\to M \otimes_A N$ defined by $t(\pi \otimes A)=\pi (\id_M \otimes
s)$, where $\pi$ is the canonical epimorphism $M \otimes N \to M
\otimes_A N$.

One verifies that $\otimes_A$ defines a monoidal structure on
$\MOD{\C}{A}$, with unit object $F(\1)=A$.  We denote by
$\MOD{\C}{(A,\sigma)}$ this monoidal category. The functor $F_A$
admits a natural structure  of strong monoidal functor from $\C$
to $\MOD{\C}{(A,\sigma)}$, which we denote by $F_{A,\sigma}$.

A tensor functor admitting a right adjoint
defines a central coalgebra:

\begin{proposition}\label{dominant2}
Let $F : \C \to \D$ be a  tensor functor  between tensor
categories over a field $\kk$, admitting a right adjoint $R$. Then
$A=R(\1)$ has a natural structure of commutative central algebra
in $\C$, with half-braiding denoted by $\sigma$. Moreover, if $R$
is faithful exact, then  $\MOD{\C}{(A,\sigma)}$ is a tensor
category and we have a tensor equivalence $K :  \D \to
\MOD{\C}{(A,\sigma)}$ such that the following triangle of tensor
functors commutes up to tensor isomorphism:
$$\xymatrix{\C \ar[r]^F \ar[rd]_{F_A} & \D \ar[d]^K\\ & \MOD{\C}{(A,\sigma)}}$$
\end{proposition}

\begin{definition} The commutative central algebra $(A,\sigma)$ associated with a tensor functor $F$ admitting a right adjoint is called the \emph{induced central algebra of $F$.}
\end{definition}

\begin{proof}
Let $F : \C \to \D$ be a strong monoidal functor between rigid
categories, and let $R : \D \to \C$ be a right adjoint of $F$.
Note that $R$ is unique up to unique isomorphism. Then the
adjunction $F \vdash R$ is monoidal, which means that $R$ has a
natural structure of monoidal functor such that the adjunction
morphisms are monoidal. Considering the opposite monoidal
categories $\C^\opp$ and $\D^\opp$ (with opposite composition and
tensor products, we have a comonoidal  adjunction $R^\opp \vdash
F^\opp$, which is in fact a Hopf adjunction because $\C^\opp$,
$\D^\opp$ are rigid. The induced coalgebra of this Hopf adjunction
is by \cite[Theorem 6.5]{blv}  a cocommutative central coalgebra
in $\C^\opp$, that is, a commutative central algebra $(A,\sigma)$
in $\C$.

As an algebra, $A=R(\1)$ with product $R_2(\1,\1)$ and unit $R_0$,
where $(R_2,R_0)$ denotes the monoidal structure of $R$. The
half-braiding $\sigma : A \otimes \id_\C \iso \id_\C \otimes A$ is
defined by the following commutative diagram:
$$
\xymatrix@C5em{
R\1 \otimes X \ar[rr]^{\sigma_X}\ar[rd]^\sim \ar[d]_{\id_{R\1} \otimes \eta_X}&& X \otimes R\1\ar[d]^{\eta_X \otimes \id_{R\1}}\ar[dl]_\sim\\
R\1 \otimes RFX\ar[r]_{R_2(\1,FX)}& RFX & \ar[l]^{R_2(FX,\1)}RFX
\otimes R\1, }
$$
where $X$ is an object of $\C$, $\eta_X : X \to RF X$ is the
adjunction unit, and the slanted arrows are the Hopf isomorphisms
of the Hopf adjunction, see \cite{blv}.

Now assume $R$ is faithful exact. In particular $R$ and $R^\opp$
are conservative. By \cite[Theorem 6.6]{blv}, $\D$ is monoidal
equivalent to $\MOD{\C}{(A,\sigma)}$ via the \kt linear functor $K
: \D \to \MOD{\C}{(A,\sigma)}$, $Y \mapsto (R(Y), R_2(Y,\1))$, and
$F_A \simeq KF$ as  tensor functors.
\end{proof}


\begin{example}\label{exa-fib}
Let $H$ be a finite dimensional Hopf algebra over a field $\kk$.
The forgetful functor $U : \CoRep{H} \to \vect_\kk$ admits a right
adjoint $R=? \otimes H$. The induced central algebra $(A,\sigma)$
is a commutative algebra in $\Z(\Rep{H})$, that is, a commutative
algebra in $\Rep{D(H)}$.  As an algebra in $\CoRep{H}$,
$A=H$ with right coaction $\Delta$. The half-braiding $\sigma$ is
defined, in Sweedler's notation, by
$$\sigma_{(V,\partial)} : \prettydef{A \otimes V &\to V \otimes A\\h \otimes x &\mapsto x_{(0)} \otimes S(x_{(1)})\, h\, x_{(2)}}$$
for any right $H$-comodule $(V,\partial)$.
We have $\MOD{\CoRep{H}}{(A,\sigma)} \simeq \vect_\kk$ as tensor categories.
\end{example}

\begin{example}\label{exa-hopfdom} Let $f: H \to H'$ be a surjective morphism between finite dimensional Hopf algebras over a field $\kk$, and denote by
 $\ff= f_* : \CoRep{H} \to \CoRep{H'}$  the dominant tensor functor defined by $f$. It has a right adjoint $R=? \,\square^{H'} H$, which is exact because $H$ is $H'$-coflat.
The induced central algebra
$(B,\sigma')$ of $\ff$ is a commutative algebra in $\Z(\CoRep H)$.
We have $B=R(\1)=\kk\, \square^{H'}H= H^{\co H'} \subset H$, where $H$ is seen as a commutative central algebra of $\CoRep{H}$ (see Example~\ref{exa-fib}).
According to Proposition~\ref{dominant2}, we have a tensor equivalence:
$$\MOD{\CoRep{H}}{(B,\sigma')} \simeq \CoRep{H'}.$$
See also \cite[Theorem II]{schn}.
\end{example}

Commutative central algebras define tensor functors.

\begin{proposition}\label{dominant1}
Let $\C$ be a tensor category over a field $\kk$, and let $(A,\sigma)$
be a commutative central algebra of $\C$ such that $\Hom_\C(\1,A)=\kk$. Then:
\begin{enumerate}
\item
If $\MOD{\C}{(A,\sigma)}$ is rigid, it is a tensor category and
the free module functor $F_{A,\sigma} : \C \to \MOD{\C}{(A,\sigma)}$ is a dominant
tensor functor, whose induced central algebra is $(A,\sigma)$;
\item If $A$ is semisimple as an algebra in $\C$, then $\MOD{\C}{(A,\sigma)}$ is rigid;
\item If $\C$ is a fusion category, $A$ is semisimple and
$\kk$ is algebraically closed, then $\MOD{\C}{(A,\sigma)}$ is a fusion
category.
\end{enumerate}
\end{proposition}

\begin{proof}
Since $(A,\sigma)$ is a commutative algebra in $\Z(\C)$,
$\MOD{\C}{A}$ admits a monoidal structure denoted by
$\MOD{\C}{(A,\sigma)}$, with tensor product $\otimes_A$ and unit
object $F(\1)=A$, hence a strong monoidal functor  $F_{A,\sigma} :
\C \to \MOD{\C}{(A,\sigma)}$. Set $\D=\MOD{\C}{(A,\sigma)}$.

The category $\D$ is abelian \kt linear, it has  finite
dimensional Homs, and objects have finite length in $\D$. Moreover
its tensor product $\otimes_A$ is \kt bilinear, and
$\End_\D(\1)=\Hom_{\MOD{\C}{A}}(A,A)\simeq \Hom_\C(\1,A)\simeq
\kk$ by assumption. If $\D$ is rigid, it is a tensor category and
$F_{A,\sigma}$ is a dominant tensor functor. Its right adjoint $R$
is the forgetful functor. Its induced central algebra is
$R(\1)=A$, with the half-braiding defined in
Proposition~\ref{dominant2}, which is in fact $\sigma$, hence Part
(1).

Since $F_A$ is strong monoidal, all objects of the form $F_A(X)$ have a left and a right dual. If $A$ is semisimple, then any object
of $\D$ is a direct factor of $F_A(X)$ for some $X$ in $\C$, so it has also a left and a right dual, so $\D$ is rigid, hence Part (2).

If $\kk$ is algebraically closed, $\C$ is a fusion category, and $A$ is semisimple, then
$\D$ is semisimple and finite because $F_A$ is dominant, so it is a fusion category, hence Part (3).
\end{proof}

\begin{corollary}\label{suff-dominant} Let $\C$ be a fusion category over an algebraically field $\kk$.

Then the following data are equivalent:
\begin{enumerate}[(A)]
\item A commutative central algebra $(A,\sigma)$ of $\C$ such that $A$ is a semisimple algebra in $\C$ and $\Hom_\C(\1,A)=\kk$;
\item A dominant tensor functor  $\ff: \C \to \D$, where $\D$ is a fusion category over $\kk$.
\end{enumerate}

\end{corollary}

\begin{proof}
According to Proposition~\ref{dominant1}, a commutative central
algebra $(A,\sigma)$ as in (A) gives rise to a dominant tensor
functor $F_A : \C \to \MOD{\C}{(A,\sigma)}$, and
$\MOD{\C}{(A,\sigma)}$ is a fusion category over $\kk$.
Conversely, let $\D$ be a fusion category and let $F : \C \to \D$
be a dominant tensor functor. Then $F$ admits a right adjoint $R$,
which is exact because $\D$ is semisimple, and faithful because
$F$ is dominant. Thus we may apply Proposition~\ref{dominant2}.
Let $(A,\sigma)$ be the induced central algebra of $F$. Then $F$
is equivalent to $F_A$, in the sense that there exists a tensor
equivalence $K : \D \to \MOD{\C}{(A, \sigma)} $ such that $F_A
\simeq K F$ as tensor functors.
\end{proof}

%
%

%

\subsection{Normal functors and trivializing algebras}

We have seen that dominant tensor functors between fusion
categories are classified by their induced central algebras. We
now characterize similarly normal tensor functors between fusion
categories.

Let $A$ be an algebra in a tensor category $\C$. We say that
\emph{$A$ trivializes an object $X$ of $\C$} if $F_A(X) \simeq
F_A(\1)^n$ for some natural integer $n$, and  $A$ is
\emph{self-trivializing} if it trivializes its underlying object.

\begin{proposition}\label{toto}
Let $F : \C \to \D$ be an exact tensor functor between tensor
categories admitting an exact right adjoint, and let $(A,\sigma)$
be its induced central algebra (in $\C$). Then $F$ is normal if
and only if the algebra $A$ is self-trivializing.  If such is the
case, $\KER_\ff = \langle A\rangle \subset \C$ so that we have an
exact sequence of tensor categories:
$$\langle A \rangle \toto \C \toto \D.$$
\end{proposition}

\begin{proof} Denote by $R$ a left adjoint of $F$, so that $A$ is isomorphic to $R(\1)$. According to Proposition~\ref{func-normal},
$\ff$ is normal $\iff$ $\ff R(\1)$ is trivial $\iff$ $\ff_A(A)$ is trivial by
Proposition~\ref{dominant2}.  Now assume $A$ is self-trivializing.
Since $\KER_F$ is an abelian subcategory  of $\C$ containing $A$
and stable under subobjects and quotients, it contains $\langle A
\rangle$ by definition. Conversely if $X$ is in
$\KER_\ff=\KER_{F_{A,\sigma}}$, then $X\otimes A \simeq A^n$ and
$\1 \into A$ so $X \into A^n$, hence $X$ is in $\langle
A \rangle$.
\end{proof}

\begin{corollary}
 An exact sequence of tensor categories $\C' \toto \C \overset{F}\toto \C''$  such that $F$ has a faithful right adjoint is equivalent to
 the exact sequence of tensor categories
 $$\xymatrix{\langle A \rangle \ar[r]& \C \ar[r]^-{F_{A,\sigma}}&\MOD{\C}(A,\sigma)},$$ where $(A,\sigma)$ denotes the induced central algebra of $F$.
\end{corollary}

\begin{proposition} Let $\ff : \C \to \D$ be a tensor functor between  fusion categories, and let $(A,\sigma)$ be its induced central algebra.
Then any simple object of $\KER_\ff$ is a direct factor of $A$.
The functor $\ff$ is normal if and only if $A$ is
self-trivializing, and if such is the case, then we have
$$A \simeq \bigoplus_{X \in \Lambda_{\KER_\ff }} X^{\FPdim X},$$ and in particular $\FPdim(\KER_\ff)= \FPdim(A)$.
\end{proposition}

\begin{proof} The functor $\ff$ can be decomposed as $\ff=i \ff_0$,
where $i: \C  \subset \D$ is the inclusion of the dominant image
of $\E$, that is, the full subcategory of $\D$ whose objects are
direct summands of elements of the image of $\ff$, and $\ff_0 : \C
\to \E$ is the dominant tensor functor induced by $\ff$. Then
$\ff$ is normal if and only if $\ff_0$ is normal, and $(A,\sigma)$
is the induced central algebra of $\ff_0$, so by
Proposition~\ref{toto}, $\ff$ is normal if and only if $A$ is
self-trivializing.

For $X \in \Lambda_\C$ we have by adjunction
$\Hom_\D(F(X),\1)=\Hom_\C(X,A)$. In particular if  $F$ trivializes
$X$ then $X$ is a direct factor of $A$. Assume  that $\ff$ is
normal. Then $\ff$ trivializes $A$, so
for any $X$ simple direct factor of $A$, we have $F(X) \simeq
\1^n$, with $n=\FPdim X$.  Also, $n=\dim \Hom(F(X),\1)=\dim
\Hom(X,A)=m_{X,A}$. Hence $A= \sum_X X^{\FPdim X}$, where $X$
ranges over the set of classes of simple factors of $A$, that is,
$\Lambda_{\KER_\ff}$.
\end{proof}

\begin{corollary}
Let $\C$ be a  finite tensor category.  The following are
equivalent:
\begin{enumerate}[(A)]
\item Fiber functors for $\C$;
\item Commutative central algebras $(A,\sigma)$ of $\C$ such that $A$  trivializes all objects of $\C$ and satisfies  $\Hom_\C(\1,A)=\kk$.
\end{enumerate}
\end{corollary}

\begin{proof}
A fiber functor for $\C$ is just a tensor functor $\omega: \C \to
\vect_\kk$. Such a tensor functor is automatically dominant and
normal; it admits a right adjoint $R$, $\C$ being finite; and $R$
is faithful (because $\omega$ is dominant) and exact, $\vect_\kk$
being semisimple. Thus by Proposition~\ref{dominant2} $\omega$ it
is classified by its induced central algebra $(A,\sigma)$.
Conversely, a commutative central algebra $(A,\sigma)$ which
trivializes all objects of $\C$ and satisfies $\Hom(\1,A)=\kk$.
Then $\D=\MOD{\C}(A,\sigma)$ is tensor equivalent to $\vect_\kk$,
and
$$\omega :
\prettydef{ \C& \to &\vect_\kk,\\X &\mapsto
&\Hom_\D(\1,F_{A,\sigma}(X))\simeq \Hom_\C(\1, X \otimes A) }$$ is
a fiber functor for $\C$.
\end{proof}

\begin{remark}
In particular, let $H$ be  a finite-dimensional Hopf algebra, with
product $m$, unit $u$, coproduct $\Delta$ and counit
$\varepsilon$. The forgetful functor $\U: \CoRep(H) \to \vect_\kk$
is a fiber functor for $\CoRep(H)$.
The induced central algebra $(A,\sigma)$ of $\U$  is $H_H$,
seen as a central algebra of $\CoRep{H}$ (see Example~\ref{exa-fib}), and we have $\U \simeq
\Hom_{\CoRep{H}}(\1, ? \otimes H_H)$ as fiber functors.
\end{remark}

\subsection{Equivariantizations in terms of central commutative algebras}\label{actionofG}

Let $\ff: \C \to \D$ be a dominant functor between tensor
categories admitting an exact right adjoint, and let $(A,\sigma)$
be its induced central algebra, so that $\D \simeq
\MOD{\C}{(A,\sigma)}$. Denote by  $G=\Aut(A,\sigma)$ the group of
automorphisms of the central algebra $(A,\sigma)$ (that is,
algebra automorphisms compatible with the half-braiding $\sigma$).
The group $G$ acts on the category $\MOD{\C}{(A,\sigma)}$ by
tensor autoequivalences (see \cite[Proposition 2.10]{eno2}),
setting:
$$\rho^g: \prettydef{
\MOD{\C}{(A,\sigma)} & \to &\MOD{\C}{(A,\sigma)}\\
(X,r)& \mapsto &\bigl(X,r(\id_X \otimes a^{-1}_g)\bigr).}$$
Moreover, for each $g \in G$, let  $u^g = \id_X \otimes a(g) :
\rho^g \ff_A(X) \iso \ff_A(X)$. This defines a tensor functor over
$\D$:
$$\ff^G_A : \prettydef{\C &\to &\mod{\C}{(A,\sigma)}^G \\ X & \mapsto & (\ff_A(X),(u^g)_{g \in G}).}$$

Hence, via the tensor equivalence $\D \simeq
\MOD{\C}{(A,\sigma)}$, we get  an action of $\Aut(A,\sigma)$ on
the tensor category $\D$ and a tensor functor $\ff^G : \C \to
\D^G$ such that the following triangle of tensor functors
commutes:
$$\xymatrix{
\C\ar[rr]^{\ff^G}\ar[rd]_\ff && \D^G\ar[ld]^{\U^G}\\
& \D & }$$

\begin{lemma}\label{glke} Assume that $\ff$ is normal. Denote by $H$ its induced Hopf algebra,
and by $G(H)$ the group of grouplike elements of $H$. Then
$\Aut(A,\sigma)$ is isomorphic to a subgroup of $G(H)$, and in
particular its order  divides  $\dim_\kk H= \dim_{ \kk }
\End_\C(A)$.\end{lemma}

\begin{proof} Recall that $\ff$ is monadic, and, denoting by $T$ its monad,  we have a canonical tensor equivalence $\C \iso \D^T$
which sends $(A,\sigma)$ to the left dual of the induced central
coalgebra $(\hat{C}, \hat\sigma )$ of $T$. In particular,
$\Aut(A,\sigma)$ is isomorphic to the group of automorphisms of
$(\hat{C},\hat{\sigma})$, which is a subgroup of the group
$\Aut(\hat{C})$ of automorphisms of the coalgebra $\hat{C}$. Now,
as a coalgebra of $\D$, $\hat{C}=(T(\1),T_2(\1,\1),T_0)$, with
$T$-action $\mu_\1$. Since $\ff$ is normal, we have $T(\1) \simeq
H \otimes \1$, and, via the canonical tensor equivalence
$\langle\1 \rangle \simeq \vect_\kk$, $\hat{C}$ is just the
coalgebra $H$ with $H$ acting by left multiplication,  whose group
of automorphisms is $G(H)$. Thus, $\Aut(A,\sigma) \subset G(H)$,
and since $H$ is free as a $\kk G(H)$-module, then the order of
$G(H)$ divides $\dim_\kk H$. We also have by adjunction $H \simeq
\Hom_\D(\1,T(\1))\simeq \End_{\D^T}(\hat{C}) \simeq \End_\C(A) $,
hence the lemma is proved.
\end{proof}

\begin{definition}
Let $X$ be a trivial object of a tensor category $\C$. There is a
unique half-braiding $\tau : X \otimes ? \to ? \otimes X$ such
that $(X,\tau)$ is a trivial object in the center $\Z(\C)$ of $\C$.  This
half-braiding is called the \emph{trivial half-braiding for $X$}.
\end{definition}

\begin{lemma}\label{lem-coco}
Let $\ff: \C \to \D$ be a dominant tensor functor between tensor
categories admitting a right adjoint, let $(A,\sigma)$ be its
induced central algebra and $T$ its monad. Denote by $\tau$ the
trivial half-braiding of the trivial object $\ff(A)$ of $\D$. Then
the following assertions are equivalent:
\begin{enumerate}[(i)]
\item The normal Hopf monad
$T$ is cocommutative;
\item For any $X$ object of $\C$, the following triangle is commutative:
$$\xymatrix{
T(X) \ar[r]^-{T_2(\1,X)} \ar[d]_{T_2(X,\1)} & T(\1) \otimes T(X) \ar[ld]^{\hat{\tau}_{T(X)}}\\
 T(X) \otimes T(\1) & }$$
\item For any object $X$ of $\C$, $\ff(\sigma_X)=\tau_{\ff(X)}$.
\end{enumerate}
\end{lemma}
\begin{proof} Since $T(\1)$ is trivial, (ii) is equivalent to:
$$\mbox {for all $g : T(\1) \to \1$,}\quad (\id_{T(X)} \otimes g)T_2(X,\1)=(\id_{T(X)} \otimes g)\hat{\tau}_{T(X)}T_2(\1,X),$$
and $(\id_{T(X)} \otimes g)\hat{\tau}_{T(X)}T_2(\1,X)=(g \otimes
\id_{T(X)}) T_2(\1,X)$  by functoriality, $\hat{\tau}$ being a
trivial half-braiding, hence (i) $\iff$ (ii).

Now let us prove (ii) $\iff$ (iii). Since $\ff$ is monadic, we may
assume that $\C=\D^T$, $\ff$ being the forgetful functor $\U :
\D^T \to \D$. Let $(\hat{C},\hat{\sigma})$ be the induced central
coalgebra of $T$. Recall that  $(\hat{C},\hat{\sigma})$ is dual to
$(A,\sigma)$. Thus, denoting by $\hat{T}$  the trivial
half-braiding of $\U \hat{C} = T(\1)$, (iii) is equivalent to
(iii)': $\ff_{\hat{\sigma}}=\hat{\tau}_\ff$. Since $T$ is a Hopf
monad, for every $T$-module $(M,r)$ we have isomorphisms
\begin{align*}
{\mathbb H}^l_{\1,(M,r)}&=(\id_{T(\1)} \otimes r)T_2(\1,M) : T(M) \iso T(\1) \otimes M, \\
{\mathbb H}^r_{(M,r),\1}&=(r \otimes\id_{T(\1)})T_2(M,\1) : T(M)
\iso M \otimes T(\1),
\end{align*}
and by definition $\hat{\sigma}_{M,r}={\mathbb
H}^r_{(M,r),\1}{\mathbb H}^{l\, -1}_{\1,(M,r)}$.

From the definition of $\hat{\sigma}$ and the axioms of a bimonad,
we deduce that for $X$ in $\D$ the following diagram is
commutative:
$$
\xymatrix{
&& T(\1) \otimes T(X) \ar[dd]^{\hat{\sigma}_{TX,\mu_X}}\\
T(X) \ar@/^1.3pc/[rru]^{T_2(\1,T(X))} \ar@/_1.3pc/[rrd]_{T_2(T(X),\1)} \ar[r]^{T(\eta_X)}&T^2(X) \ar[ru]_{\mathbb{H}^l_{\1,(TX,\mu_X)}} \ar[rd]^{\mathbb{H}^r_{(TX,\mu_X),\1}} &\\
&&T(X) \otimes T(\1), }
$$
from which we deduce that (iii)' \implique (ii). Conversely, if
(ii) holds the following diagram commutes:
$$
\xymatrix@R=0.7em @C=4em{
& T(\1) \otimes T(M) \ar[r]^{\id_{T(\1)} \otimes r} \ar[dd]^{\hat{\tau}_{T(M)}}  & T(\1) \otimes M  \ar[dd]^{\hat{\tau}_{M}}\\
T(M) \ar[ur]^-{T_2(\1,M)} \ar[dr]_-{T_2(M,\1)} &&\\
& T(M) \otimes T(\1) \ar[r]_{r \otimes \id_{T(\1)}}& M \otimes
T(\1), }
$$
so $\hat{\tau}_{M}={\mathbb H}^r_{(M,r),\1}{\mathbb H}^{l\,
-1}_{\1,(M,r)}=\hat{\sigma}_{(M,r)}$, hence (iii)' holds. Hence
the Proposition is proved.
\end{proof}

\begin{proposition}
Let $\ff: \C \to \D$ be a dominant tensor functor between tensor
categories admitting an exact right adjoint, and let $(A,\sigma)$
be its induced central algebra. Then $\ff$ is tensor equivalent to
an equivariantization if and only if the following conditions are
met:
\begin{enumerate}
\item The tensor functor $\ff$ is normal, that is, $\ff(A)$ is trivial;
\item For all $X$ object of $\C$, $\ff(\sigma_X)=\tau_{\ff(X)}$, where $\tau$ is the trivial half-braiding of $\ff(A)$;
\item The induced Hopf algebra  $H$  of $\ff$ is split semisimple.
\end{enumerate}
If such is the case, $\Aut(A,\sigma)$ is isomorphic to $G(H)$ and
$\C \simeq \D^{\Aut(A,\sigma)}$ over $\D$.
\end{proposition}

\begin{proof} Let $T$ be the monad of $\ff$.
Assume that $\ff$ is equivalent to an equivariantization, that is,
there exists an action $\rho$ of a finite group $\Gamma$ on  $\D$
by tensor autoequivalences, and a tensor equivalence $\D^\Gamma
\iso \C$ over $\D$. We may assume that $T=T^\Gamma$, $\ff$ being
the forgetful functor $\U^\Gamma$. Then by Theorem~\ref{monact}
$T$ is normal, cocommutative, and its induced Hopf algebra is $\kk
\Gamma$. The cocommutativity of $T$ implies Condition (2) by
Lemma~\ref{lem-coco}.  The induced Hopf algebra of $\U^\Gamma$ is
$\kk \Gamma$, and we have $\Gamma=G(H)$. Let $(A,\sigma)$ be the
induced central algebra of $T^\Gamma$. According to
Lemma~\ref{glke}, the group of automorphisms of the algebra $A$ is
$\Gamma$, and its order is the dimension of $\End_\C(A)$. We show
that the group of automorphisms of $(A,\sigma)$ is $\Gamma$.
Observe first that if $\ff : \C \to \D$ is a dominant tensor
functor between tensor categories admitting an exact right adjoint
$R$, and $\phi  : \D \to \D' $ is a tensor equivalence,
then the induced central algebra of $\phi\ff$ is canonically
isomorphic to that of $\ff$, and we may therefore identify them.
Secondly, by construction of the equivariantization, we have for
each $\gamma \in \Gamma$ a canonical isomorphism $\rho^\gamma
\U^\Gamma \simeq \U^\Gamma$. This induces an isomorphism between
the induced central algebra of $\U^\Gamma$ and that of
$\rho^{\gamma}\U^\Gamma$, that is an automorphism of $(A,\sigma)$.
Thus, we have $\Aut(A,\sigma)=\Aut(A) \simeq G(H)$ and its order
is the dimension of $\End_\C(A)$.

Conversely assume Conditions (1), (2), (3) are satisfied.
Condition (2) implies that $T$ is cocommutative by
Lemma~\ref{lem-coco},  and by Theorem~\ref{converse-equiv}, $T$ is
the monad of a group action on $\C$.
\end{proof}

\subsection{Normal fusion subcategories}
Recall that every normal dominant tensor functor $\ff: \C \to
\C''$ between fusion categories gives  an exact sequence of fusion
categories $\KER_\ff \to \C \overset{\ff}\to \C''$. This motivates
the following definition.

\begin{definition} Let $\C$ be a fusion category. A full fusion
subcategory $\C'$ of $\C$ is called \emph{normal} if there exists
a normal dominant functor of fusion categories $\ff: \C \to \C''$
such that $\C' \simeq \KER_\ff$. \end{definition}

The next proposition characterizes normal subcategories in term of
trivializing algebras.

\begin{proposition}\label{suff-normal} Let $\C$ be a fusion category over an algebraically closed field $\kk$, and let $\C' \subset \C$ be a full fusion subcategory.
Then $\C'$ is a normal subcategory of $\C$ if and only if there
exists  commutative central algebra $(A,\sigma)$ of $\C$
satisfying the following conditions:
\begin{enumerate}
\item $A$ is a semisimple algebra in $\C$;
\item $\Hom(\1,A) \simeq \kk$;
\item $A$ belongs to $\C'$ and trivializes all objects of $\C'$.
\end{enumerate}
In that case, $\MOD{\C}{A}$ is a fusion category over $\kk$ and we
have an exact sequence of fusion categories $$\C' \to \C
\overset{F}\to \MOD{\C}{(A,\sigma)}.$$
\end{proposition}

\begin{proof}
This results immediately from Corollary~\ref{suff-dominant} and
Proposition~\ref{suff-normal}.
\end{proof}

\subsection{Simple fusion categories}
We define  a simple fusion category in terms of exact sequences,
as follows.


\begin{definition}\label{simple} A fusion category $\C$ is \emph{simple} if $\C$ is not tensor equivalent to $\vect_\kk$ and
for every exact sequence of fusion categories $$\C' \to \C \to
\C'',$$ either $\C'$ or $\C''$ is tensor equivalent to
$\vect_\kk$.
\end{definition}

\begin{remark} A fusion category
$\C$ is simple if and only it $\C \not\simeq \vect_\kk$ and for
any normal dominant tensor functor $\ff: \C \to \D$, we have $\D
\simeq \vect_\kk$ or $\ff$ is an equivalence. This is because such
a functor $\ff$ fits in an exact sequence $\KER_\ff \to \C \to
\D$.
\end{remark}

Note that a different notion of a simple fusion category was
introduced in \cite[Definition 9.10]{eno2}:  a fusion category is
simple in the sense of \cite{eno2} if it has no proper fusion
subcategories, that is,
$\vect_\kk$ and $\C$ are the only replete fusion subcategories of $\C$. 
The next proposition compares this definition and our
definition~\ref{simple}.

\begin{proposition}\label{simple-prim} If a fusion category $\C$ is simple in the sense of \cite{eno2}, it is simple in the sense of
Definition \ref{simple}. \end{proposition}

\begin{proof} Let $\C' \overset{f}\to \C \overset{\ff}\to \C''$ be
an exact sequence of fusion categories, with $\C \not\simeq
\vect_\kk$. If $\C$ has no proper fusion subcategories, then
either $\C' \simeq \vect_\kk$ or $f: \C' \to \C$ is an
equivalence. One concludes with Lemma \ref{trivial-sec}.
\end{proof}

The converse of Proposition \ref{simple-prim} is false. Indeed, we
deduce directly from the results of Section~\ref{pointed}:

\begin{proposition}\label{prop-point-simp} Let $\C$ be a pointed fusion category, with Picard group $G$. Let $\alpha \in H^3(G,\kk^\times)$ be the cohomology class defining $\C$, so that $\C \simeq \C(G,\alpha)$.
Then
\begin{enumerate}
\item The category $\C$ is simple in the sense of \cite{eno2} if and only if $G$ is a cyclic group of prime order;
\item The category $\C$ is simple if and only if there is no proper distinguished subgroup $H \lhd G$ such that the restriction of $\alpha$ to $H$ is trivial.
\end{enumerate}
In particular if $G$ is simple, $\C$ is simple, but it is not
simple in the sense of \cite{eno2} except if $G \simeq Z_p$, $p$
prime.
\end{proposition}

However, both notions coincide when restricted to categories of
representations of finite groups, as follows from the next
proposition.



\begin{proposition} Let $G$ be a finite group such that $\kk$ is a splitting field for $G$ and $\cha(k)$ does not divide the order of $G$.
Then the following assertions are equivalent:
\begin{enumerate}[(i)]
\item The group $G$ is simple.
\item The fusion category $\C(G, 1)$ is simple.
\item The fusion category $\rep G$  is simple.
\item The fusion category $\rep G$ is simple in the sense of \cite{eno2}.
\end{enumerate}
\end{proposition}

\begin{proof} The equivalence between (i) and (ii) is a special case of Proposition~\ref{prop-point-simp}.
We have (iii) $\Rightarrow$ (i) by Corollary \ref{sec-groups}. Now
if $G$ is a simple group, then  $\rep G$  has no proper
fusion subcategories, hence (i) $\Rightarrow$ (iv), and lastly
(iv) $\Rightarrow$ (iii) by Proposition \ref{simple-prim}, hence
the Proposition is proved.
\end{proof}

\section{Braided fusion categories of odd square-free
dimension}\label{ribbon-sqfree}


\subsection{Braided categories of odd Frobenius-Perron dimension}

Let $\C$ be a fusion category over a field $\kk$. Recall that  the
\emph{Frobenius-Perron dimension} of $\C$ is  $\FPdim \C : =
\sum_{X \in \Lambda_\C} (\FPdim X)^2$.

On the other hand, the \emph{global dimension} of $\C$ is defined
as $\dim \C : = \sum_{X \in \Lambda_\C} |X|^2$, where $|X|^2 \in
\kk^\times$ denotes the squared norm of the simple object $X$, see
\cite[Definition 2.2]{ENO}. When $\FPdim \C = \dim \C$, $\C$ is
called \emph{pseudo-unitary} .

If $\C$ is a pivotal fusion category, one defines the the
categorical left dimension $\dim^l X$ and right dimension $\dim^r
X$ of an object $X$ of $\C$. For $X$ simple, one has $|X|^2 =
\dim^l X\,\dim^r X$, so that $\dim \C = \sum_{X \in \Lambda_\C}
\dim^l X \,\dim^r X$. The category $\C$ is \emph{spherical} if
left and right dimensions coincide; in that case, they are denoted
by $\dim$, and we have: $\dim_\C=\sum_{X \in \Lambda_\C} (\dim
X)^2$.

%
%

Assume $\kk=\Comp$. Then $\C$ is \emph{pseudo-unitary} if $\FPdim
\C=\dim \C$. If such is the case, then by \cite[Proposition
8.23]{ENO}, $\C$ admits a unique spherical structure with respect
to which the categorical dimensions of simple objects are all
positive. We call it the \emph{canonical spherical structure}. For
this structure, the categorical dimension of an object coincides
with its Frobenius-Perron dimension.
%
%

If  $\C$ is a fusion category on $\Comp$ such that $\FPdim \C$ is
an integer, then $\C$ is pseudo-unitary by \cite[Proposition
8.24]{ENO}. Moreover, $\FPdim \C'$ is an integer for any full
fusion subcategory $\C' \subseteq \C$, because $\FPdim \C' =
\sum_{X \in \Lambda_{\C'}} (\FPdim X)^2$ and for each $X$, $\FPdim
X$ is the square root of a natural integer by \cite[Proposition
8.27]{ENO}. In particular, every full fusion subcategory of $\C$
is pseudo-unitary.

\begin{lemma}\label{sym}  Let $\C$ be a symmetric fusion category over a field $\kk$ whose Frobenius-Perron dimension is an odd natural integer. Then any balanced structure on $\C$ is trivial.
\end{lemma}


\begin{proof}
Since $\C$ is symmetric, a balanced structure $\theta$ on $\C$ is
a monoidal automorphism of $\id_\C$. Since $\theta^2 =
\id_{\C}$,  $\theta$ defines a $\{\pm 1\}$-graduation on $\C$,
with $\C_1 \subset \C$ being the full tensor subcategory of
objects $X$ such that $\theta_X=1$. If $\C_1 \neq \C$, we have
$\FPdim \C= 2 \FPdim \C_1$, which contradicts the fact that
$\FPdim \C$ is odd, hence $\C_1=\C$ and $\theta=1$.
\end{proof}


\begin{lemma}\label{mod-odd} Let $\C$ be a braided fusion category over $\Comp$ such that
$\FPdim \C$ is an odd natural integer. Then $\C$, endowed with its
canonical spherical structure, is modularizable.
\end{lemma}

\begin{proof} The category $\C$ is pseudo-unitary.
Equipped with its canonical spherical structure, it is  a
premodular category and we have $\dim X =\FPdim X \ge 0$ for any
object $X$ of $\C$. Let $\T \subset\C$ be the full
tensor subcategory of transparent objects of $\C$; it is a
symmetric fusion category \cite[Section 2]{bruguieres}. By
\cite[Theorem 7.2]{aeg}, the categorical dimensions of the objects
of $\T$ are integers. Since $\FPdim \C$ is a natural integer, so
is $\FPdim \T$. By \cite[Proposition 8.15]{ENO}, $\FPdim \C /
\FPdim \T$ is an algebraic integer, so $\FPdim \T$ divides $\FPdim
\C$. Thus $\FPdim \T$ is an odd natural integer as well. By Lemma
\ref{sym}, we have $\theta_X=\id_X$ for all $X$ in $\T$. By
\cite[Th\'{e}or\`{e}me 3.1]{bruguieres}, $\C$ is modularizable,
hence the Lemma is proved.
\end{proof}


\subsection{Proof of Theorem  \ref{classif}}

In order to prove the theorem, we may assume that the ground field
is $\Comp$. Indeed, if $\C$ is a fusion category over a field
$\kk$ of characteristic $0$, then  $\C$ is defined over the
algebraic closure $\overline{\mathbb{Q}} \subset \kk$ of
$\mathbb{Q}$, which we may imbed into $\Comp$. We deal first with
the modular case.

\begin{lemma}\label{mod-sqf} Let $\C$ be a modular category whose  Frobenius-Perron  dimension is a square-free odd integer $N$.
Then there exists an abelian group $G$ of order $N$ such that $\C$
is equivalent to the category of $G$-graded vector spaces $\C(G,
1)$.
\end{lemma}

\begin{proof}
By \cite[Theorem 2.11 (ii)]{eno2}, for $X  \in \Lambda_{\C}$ we
have  $\FPdim X = 1$, so $X$ is invertible, hence $\C$ is a
pointed category. Thus $\C=\C(G,\alpha)$ for some finite group $G$
and some cohomology class $\alpha \in H^3(G,\kk^\times)$.

Now braided and ribbon structures on a pointed fusion category
$\C=\C(G,\alpha)$ are classified in \cite[7.5]{FK} in terms of
group cohomology. (See also \cite[2.4]{DGNO}.) The existence of a
braiding implies that $G$ is abelian. Moreover by
\cite[Proposition 7.5.3 iii)]{FK}, given a balanced structure
$\theta$ on $\C$, the class $\alpha$ is trivial if and only if
$\theta$ is equal to $1$ on the subgroup ${}_2G : = \{ g\in G: \,
g^2 = 1 \}$.

In the present case, we conclude that $G$ is an abelian group of
odd order $N$, so $_2G=1$ and therefore $\alpha$ is trivial, hence
the Lemma holds.
\end{proof}

\begin{lemma}
A braided fusion category $\C$ whose  Frobenius-Perron dimension
is an odd square-free integer $N$ admits a fiber functor.
\end{lemma}

\begin{proof}
By Lemma \ref{mod-odd} the category $\C$, endowed with its canonical spherical
structure, is modularizable. In particular the full subcategory
$\T \subset \C$ of transparent objects of $\C$ is tannakian; we
have $\T \simeq  \rep G$ as symmetric tensor categories,
$G$ being a finite group, and we have an exact sequence
$$\xymatrix{\T \ar[r] & \C \ar[r]&\widetilde \C},$$ where $\widetilde \C$
is a modular category, $G$ acts on $\tilde{\C}$ by braided tensor
autoequivalences, and $\C\simeq \tilde{\C}^G$ as braided tensor
categories, see Example \ref{exmodable}.

By Proposition \ref{fp-dim}, $\FPdim \widetilde \C=N/ \FPdim \T$,
and $\FPdim \T$ is a natural integer, so $\FPdim \widetilde \C$ is
an odd square-free integer. By Proposition \ref{mod-sqf},
$\widetilde \C$  admits a fiber functor, and so does $\C$.
\end{proof}

So $\C$ is tensor equivalent to  $\Rep{H}$, where $H$ is a
quasitriangular Hopf algebra whose dimension is odd and
square-free. By \cite[Theorem 1.2]{qt-quotient}, such a Hopf algebra  is
isomorphic to a group algebra, hence the theorem is proved. \qed

\bibliographystyle{amsalpha}

\begin{thebibliography}{A}



\bibitem{aeg} {\sc N. Andruskiewitsch}, {\sc P. Etingof} and {\sc S. Gelaki},
\emph{Triangular Hopf algebras with the Chevalley property},
Michigan Math. J. {\bf 49} (2001), 277--298.

\bibitem{agaitsgory} {\sc S. Arkhipov} and {\sc D. Gaitsgory},
\emph{Another realization of the category of modules over the
small quantum group}, Adv. Math. \textbf{173} (2003), 114--143.

\bibitem{SGA4} {\sc M. Artin, A. Grothendieck} and \sc{J.-L. Verdier} \emph{ Théorie des topos et cohomo-
logie étale des schémas (SGA4)}, Lecture Notes in Mathematics, Vol. 269, 270,
305, Springer-Verlag, 1972-1973.

\bibitem{bruguieres} {\sc A. Brugui\`{e}res}, \emph{Cat\' egories pr\'
emodulaires, modularisations et invariants des vari\' et\' es de
dimension $3$}, Math. Ann. {\bf 316} (2000), 215--236.

\bibitem{bv} {\sc A. Brugui\`{e}res} and \textsc{A. Virelizier},
\emph{Hopf monads}, Adv. Math. \textbf{215} (2007), 679--733.

\bibitem{bv-double} {\sc A. Brugui\`{e}res} and \textsc{A. Virelizier},
\emph{The double of a Hopf monad}, preprint \texttt{arXiv:\-
0812.2443}- to appear in Trans. Amer. Math. Soc.

\bibitem{blv} {\sc A. Brugui\`{e}res}, \textsc{S. Lack} and \textsc{A. Virelizier}, \emph{Hopf monads on monoidal categories},
preprint \texttt{arXiv:\-1003.1920}

\bibitem{curtis-reiner} \textsc{C. Curtis} and \textsc{I. Reiner}, \textit{Methods of Representation Theory}
Vol. I, Wiley Interscience, New York, 1990.

\bibitem{DNR} S. Dascalescu, C. Nastasescu, S. Raianu,
\emph{Hopf algebras. An introduction}, Pure and Applied
Mathematics \textbf{235}, Marcel Dekker, New York (2001).


\bibitem{DM} {\sc P. Deligne} and \textsc{J. Milne}, \emph{Tannakian
Categories}, Lect. Notes Math. \textbf{900}, 101--228, 1982.

\bibitem{deligne} {\sc P. Deligne},
\emph{Cat\' egories tannakiennes}, Progr. Math. \textbf{87},
Birkh\" auser, Boston, 1990.

\bibitem{deligne2} \textsc{P. Deligne}, \emph{Cat\' egories tensorielles}, Mosc. Math. J. \textbf{2} (2002), 227--248.

\bibitem{DGNO} {\sc V. Drinfeld}, \textsc{S. Gelaki}, \textsc{D. Nikshych} and
\textsc{V. Ostrik}, \emph{Group-theoretical properties of
nilpotent modular categories}, preprint \texttt{arXiv:0704.0195}
(2007).

\bibitem{EO} {\sc P. Etingof} and {\sc V. Ostrik},
\emph{Finite tensor categories}, Mosc. Math. J. \textbf{4} (2004),
627--654.

\bibitem{ENO} {\sc P. Etingof}, {\sc D. Nikshych} and {\sc V. Ostrik},
\emph{On fusion categories}, Ann. Math. (2) \textbf{162} (2005),
581--642.

\bibitem{eno2} {\sc P. Etingof}, {\sc D. Nikshych} and {\sc V. Ostrik},
\emph{Weakly group theoretical and solvable fusion categories},
preprint \texttt{arxiv:math.QA/0809.303} (2008).

\bibitem{fw} {\sc E. Frenkel}, and {\sc E. Witten},
\emph{Geometric endoscopy and mirror symmetry},  Commun.
Number Theory Phys. \textbf{2} (2008),  113--283.

\bibitem{FK} {\sc J. Fr\" ohlich}, and {\sc T. Kerler},
\emph{Quantum groups, quantum categories and quantum field
theory}, Lect. Notes Math. \textbf{1542}, Springer-Verlag, Berlin,
1993.

\bibitem{GN} {\sc C. Galindo}, and {\sc S. Natale},
\emph{Simple Hopf algebras and deformations of finite groups},
Math. Res. Lett. \textbf{14} (2007), 943--954.

\bibitem{moerdijk} {\sc I. Moerdijk},
\emph{Monads on tensor categories}, J. Pure Appl. Algebra
\textbf{168} (2002), 189--208.

\bibitem{maclane} {\sc S. MacLane},
\emph{Categories for the working mathematician}, Springer-Verlag,
New York (1998).

\bibitem{masuoka} \textsc{A. Masuoka},
\emph{Example of almost commutative Hopf algebras which are not
coquasitriangular}, Lect. Notes Pure Appl. Math. \textbf{237}
(2004), 185--191.

\bibitem{nik-gel} \textsc{S. Gelaki} and \textsc{D. Nikshych},
\emph{Nilpotent fusion categories}, Adv. Math. \textbf{217}
(2008), 1053--1071.

\bibitem{kob-mas} \textsc{T. Kobayashi} and \textsc{A. Masuoka},
\emph{A result extended from groups to Hopf algebras}, Commun.
Algebra \textbf{25} (1997), 1169--1197.

\bibitem{muger} {\sc M. M\" uger}, \emph{Galois theory for braided tensor
categories and the modular closure}, Adv. Math. \textbf{150}
(2000), 151--201.

\bibitem{Mo} {\sc S. Montgomery}, \emph{Hopf algebras and
their actions on rings}, CBMS Reg. Conf. Ser. Math. \textbf{82},
Am. Math. Soc., Providence, RI,  1993.

\bibitem{muger-luminy} {\sc M. M\" uger}, \emph{Tensor categories: a selective
guided tour} (2008), preprint \texttt{arXiv:\- 0804.3587v1}.

\bibitem{ssld} {\sc S. Natale}, \emph{Semisolvability of semisimple Hopf
algebras of low dimension}, Memoirs Amer. Math. Soc. \textbf{186}
(2007).

\bibitem{qt-quotient} {\sc S. Natale}, \emph{$R$-matrices and Hopf algebra
quotients}, Int. Math. Res. Not. \textbf{2006} (2006), 1--18.

\bibitem{ext-ty} {\sc S. Natale}, \emph{Hopf algebra extensions
of group algebras and Tambara-Yamagami categories}, Algebr. and
Repr. Theory, to appear. Preprint \texttt{arXiv:0805.3172v1}.

\bibitem{NZ} {\sc W. Nichols} and \textsc{M. B. Zoeller},
\emph{A Hopf algebra freeness theorem}, Am. J. Math. \textbf{111} (1989), 381--385.

\bibitem{nik} \textsc{D. Nikshych},
\emph{Non group-theoretical semisimple Hopf algebras from group
actions on fusion categories}, Sel. Math. New Ser. \textbf{14}
(2008), 145--161.



\bibitem{schn} \textsc{H.-J. Schneider}, \emph{Principal homogeneous spaces for arbitrary Hopf algebras},
Israel J. Math. \textbf{72} (1990), 167--195.

\bibitem{schn2} \textsc{H.-J. Schneider}, \emph{Some remarks on exact sequences of quantum groups},
Comm. Alg. \textbf{21} (1993), 3337-3357.


\bibitem{tambara} \textsc{D. Tambara},
\emph{Invariants and semi-direct products for finite group actions
on tensor categories}, J. Math. Soc. Japan \textbf{53} (2001),
429--456.

\bibitem{TY} \textsc{D. Tambara} and \textsc{S. Yamagami},
\emph{Tensor categories with fusion rules of self-duality for
finite abelian groups}, J. Algebra \textbf{209} (1998), 692--707.

\end{thebibliography}

\end{document}